\documentclass[12pt]{amsart}

\usepackage{latexsym,enumerate}
\usepackage{amssymb}
\usepackage{amsfonts}
\usepackage[colorlinks]{hyperref}

\newtheorem{theorem}{Theorem}[section]
\newtheorem{lemma}[theorem]{Lemma}
\newtheorem{proposition}[theorem]{Proposition}

\theoremstyle{definition}

\theoremstyle{remark}

\numberwithin{equation}{section}

\newcommand{\GL}{{\mathrm {GL}}}
\newcommand{\PGL}{{\mathrm {PGL}}}
\newcommand{\SL}{{\mathrm {SL}}}
\newcommand{\PSL}{{\mathrm {PSL}}}

\newcommand{\GU}{{\mathrm {GU}}}
\newcommand{\PGU}{{\mathrm {PGU}}}
\newcommand{\SU}{{\mathrm {SU}}}
\newcommand{\PSU}{{\mathrm {PSU}}}

\newcommand{\GO}{{\mathrm {GO}}}

\newcommand{\SO}{{\mathrm {SO}}}

\newcommand{\Sp}{{\mathrm {Sp}}}
\newcommand{\PSp}{{\mathrm {PSp}}}

\newcommand{\Aut}{{\mathrm {Aut}}}
\newcommand{\Out}{{\mathrm {Out}}}

\newcommand{\Irr}{{\mathrm {Irr}}}
\newcommand{\IBR}{{\mathrm {IBr}}}

\newcommand{\Cl}{{\mathrm {Cl}}}

\newcommand{\QQ}{{\mathbb Q}}

\newcommand{\NN}{{\mathbb N}}

\newcommand{\GC}{\mathcal{G}}

\newcommand{\ta}{\hspace{0.5mm}^{2}\hspace*{-0.2mm}}
\newcommand{\tb}{\hspace{0.5mm}^{3}\hspace*{-0.2mm}}

\newcommand{\bC}{{\mathbf{C}}}
\newcommand{\bO}{{\mathbf{O}}}
\newcommand{\bN}{{\mathbf{N}}}
\newcommand{\bZ}{{\mathbf{Z}}}
\newcommand{\Al}{\textup{\textsf{A}}}
\newcommand{\Sy}{\textup{\textsf{S}}}

\begin{document}

\title[$p$-Regular classes and $p$-rational characters]
{$p$-Regular conjugacy classes\\ and $p$-rational
irreducible characters}

\author{Nguyen Ngoc Hung}
\address{Department of Mathematics, The University of Akron, Akron,
OH 44325, USA} \email{hungnguyen@uakron.edu}

\author{Attila Mar\'{o}ti}
\address {Alfr\'ed R\'enyi Institute of Mathematics, Re\'altanoda utca 13-15, H-1053, Budapest, Hungary}
\email{maroti.attila@renyi.hu}

\thanks{The work of the second author on the project leading to this application has
received funding from the European Research Council (ERC) under the
European Union's Horizon 2020 research and innovation programme
(grant agreement No. 741420). The second author was also supported
by the National Research, Development and Innovation Office (NKFIH)
Grant No.~K132951 and Grant No.~K115799.}




\begin{abstract} Let $G$ be a finite group of order divisible by
a prime $p$. The number of $p$-regular and $p'$-regular conjugacy
classes of $G$ is at least $2\sqrt{p-1}$. Also, the number of
$p$-rational and $p'$-rational irreducible characters of $G$ is at
least $2\sqrt{p-1}$. Along the way we prove a uniform lower bound
for the number of $p$-regular classes in a finite simple group of
Lie type in terms of its rank and size of the underlying field.
\end{abstract}

\maketitle

\begin{center}
{\it To the memory of Jan Saxl.}
\end{center}

\bigskip


\section{Introduction}

The number $k(G)$ of conjugacy classes of a finite group $G$, which
is equal to the number of complex irreducible characters of $G$, is
a fundamental invariant in group theory and representation theory.
For instance, Higman's famous conjecture is to show that if $p$ is a
prime and $G$ is a Sylow $p$-subgroup of the general linear group
$\GL_{n}(q)$, then $k(G)$ is a polynomial in $q$ with integer
coefficients. The celebrated $k(GV)$ theorem states that if $V$ is a
finite, faithful, coprime $G$-module for a finite group $G$ then the
number $k(GV)$ of conjugacy classes of the semidirect product $GV$
is at most $|V|$.

Lower bounds for the number of conjugacy classes of a finite group
have a long history. Landau \cite{Landau}, in response to a question
of Frobenius, proved in 1903 that for a given $k$ there are only
finitely many groups having $k$ conjugacy classes. This result may
be translated to a lower bound on the number of conjugacy classes of
a finite group $G$ only in terms of the order of $G$. Problem 3 of
Brauer's list of problems \cite{Brauer63} was to give a
substantially better lower bound for $k(G)$. This was solved by
Pyber \cite{Pyber92} and his bound was later slightly improved by
several authors, see
\cite{Bertram06,Keller11,Baumeister-Maroti-TongViet}. In general it
is not known whether there is a universal constant $c
> 0$ such that for every finite group $G$ we have $k(G) > c \cdot
\log|G|$.

Bounding $k(G)$ only in terms of a prime divisor $p$ of $|G|$ is
another fundamental problem. It is related to Problem 21 of Brauer
\cite{Brauer63} and a conjecture of H\'ethelyi and K\"ulshammer
\cite{Hethelyi-kulshammer} that for any $p$-block $B$ of any finite
group $G$ the number $k(B)$ of complex irreducible characters in $B$
is $1$ or is at least $2 \sqrt{p-1}$. As observed by Pyber, work of
Brauer \cite{Brauer42} implies that $k(G) \geq 2 \sqrt{p-1}$ for $G$
a finite group whose order is divisible by a prime $p$ but not by
$p^{2}$. Since then, this bound had been conjectured to be true for
all groups $G$ and all primes $p$ dividing $|G|$.

Proving $k(G)\geq 2\sqrt{p-1}$ for all $G$ and $p$ has turned out to
be a hard problem. Building on a series of relevant works by
H\'{e}thelyi-K\"{u}lshammer \cite{Hethelyi-kulshammer,HK}, Malle
\cite{Malle06}, Keller \cite{Keller09}, and
H\'{e}thelyi-Horv\'{a}th-Keller-Mar\'{o}ti \cite{Hethelyietal}, the
conjecture was finally confirmed in \cite{Maroti16}.

One of the purposes of this paper is to obtain a $p$-modular analog
of the bound $k(G)\geq 2\sqrt{p-1}$; that is, to obtain a bound for
the number of $p$-regular classes, which is also the number of
irreducible $p$-modular representations, of $G$. On the other hand,
we observe that since the class number $k(G)$ is a global
characteristic of $G$ while the bound $2\sqrt{p-1}$ depends only on
$p$, it is natural to expect that the same bound would hold for a
certain subset of conjugacy classes or irreducible characters that
are defined locally in terms of $p$. We confirm this expectation
from both perspectives: classes and characters, by considering
orders of group elements and fields of character values.

An element of a finite group $G$ is called $p$-regular if it has
order coprime to $p$. Throughout let $k_{p'}(G)$ denote the number
of conjugacy classes of $p$-regular elements in $G$ and let
$k_{p}(G)$ denote the number of conjugacy classes of
\emph{non-trivial} $p$-elements in $G$.

\begin{theorem}\label{theorem-general-bound-class}
If $G$ is a finite group and $p$ is a prime dividing the order of $G$, then
\[
k_{p}(G)+k_{p'}(G)\geq 2\sqrt{p-1}
\]
with equality if and only if $\sqrt{p - 1}$ is an
integer, $G = C_p \rtimes C_{\sqrt{p-1}}$ and $\bC_G(C_p) = C_p$.
\end{theorem}

Let $\IBR_p(G)$ denote the set of irreducible $p$-modular
representations of $G$. As $|\IBR_p(G)|=k_{p'}(G)$,
Theorem~\ref{theorem-general-bound-class} provides a somewhat
unexpected lower bound for the number of irreducible $p$-Brauer
characters, namely $|\IBR_p(G)|\geq 2\sqrt{p-1}-k_{p}(G)$, and
therefore, it can be viewed as a modular version of the bound
$k(G)\geq 2\sqrt{p-1}$ for the number of ordinary irreducible
characters of $G$.

We make the bound $|\IBR_p(G)|\geq 2\sqrt{p-1}-k_{p}(G)$ more
explicit in the case when $G$ is non-$p$-solvable.

\begin{theorem}
\label{nonsolvable} Let $p$ be a prime. Let $G$ be a
non-$p$-solvable finite group. The number $|\IBR_p(G)|$ of
irreducible $p$-Brauer characters of $G$ is larger than
$2\sqrt{p-1}$ unless possibly if $p \leq 257$. In any case,
$|\IBR_p(G)| > \sqrt{p-1}$.
\end{theorem}

Next we turn to fields of character values. For a positive integer
$n$, let $\QQ_n$ denote the cyclotomic field extending the field of
rational numbers $\QQ$ by a primitive $n$th root of unity. For $\pi$
a set of primes, we say that $\chi\in\Irr(G)$ is $\pi$-rational if
there is a positive integer $n$ coprime to every prime in $\pi$ such
that $\chi(g)\in \QQ_n$ for all $g\in G$. We denote the set of all primes different from $p$ by $p'$ and we write $p$-rational instead of $\{p\}$-rational. Let $\Irr_{\mathrm{p-rat}}(G)$ and $\Irr_{\mathrm{p'-rat}}(G)$ denote the sets of $p$-rational and $p'$-rational characters of $G$. Note that a character $\chi$ is $p'$-rational if and only if its
values are in $\QQ_{|G|_p}$. Also, $\Irr_{\mathrm{p-rat}}(G)\cap
\Irr_{\mathrm{p'-rat}}(G)$ is equal to $\Irr_{\QQ}(G)$, the set of rational irreducible
characters of $G$.

\begin{theorem}\label{theorem-general-bound-character}
If $G$ is a finite group and $p$ is a prime dividing the order of $G$, then
\[
|\Irr_{\mathrm{p-rat}}(G)\cup\Irr_{\mathrm{p'-rat}}(G)|\geq 2\sqrt{p-1}
\]
with equality if and only if $\sqrt{p - 1}$ is an
integer, $G = C_p \rtimes C_{\sqrt{p-1}}$ and $\bC_G(C_p) = C_p$.
\end{theorem}

Theorems \ref{theorem-general-bound-class} and
\ref{theorem-general-bound-character} show that, in groups of order
divisible by a prime $p$, there is a correlation between $k_{p}(G)$
and $k_{p'}(G)$ as well as $|\Irr_{\mathrm{p-rat}}(G)|$ and
$|\Irr_{\mathrm{p'-rat}}(G)|$: if one is small, the other must be large
(compared to $p$, of course). In the minimal situations where one
number is minimal/small, the bound indeed could be improved. We plan
to address this at another time.

On the way to the proofs of
Theorems~\ref{theorem-general-bound-class}, \ref{nonsolvable} and
\ref{theorem-general-bound-character}, we have to bound the number
of $p$-regular classes in finite simple groups. The following
uniform bound for simple groups of Lie type is of independent
interest and might be useful in other applications.

\begin{theorem}\label{theorem-p-regular-bound-Lietype}
If $S$ is a simple group of Lie type defined over the field of $q$
elements with $r$ the rank of the ambient algebraic group and $p$ is any prime, then
\[
k_{p'}(S)> \frac{q^r}{17r^2}.
\]
\end{theorem}

Better and more refined bounds for different types and different $p$
are given in Sections~\ref{section-linear-unitary},
\ref{section-symplectic-orthogonal} and \ref{section-exceptional}.
We remark that the problem of bounding the class number (both upper
and lower bounds) of finite groups of Lie type has been well
studied, for instance in the influential work of Fulman and
Guralnick \cite{Fulman-Guralnick12}. To provide a relative
comparison between $k_{p'}(S)$ and $k(S)$, we note that the best
general lower bound for $k(S)$ is $q^r/d$, where $d$ is the order of
the group of diagonal automorphisms of $S$.

Theorems \ref{theorem-general-bound-class}, \ref{nonsolvable}, and
\ref{theorem-general-bound-character} are proved in
Sections~\ref{section-p-regular-class},
\ref{section-Brauer-characters}, and \ref{section-p-rational},
respectively. In Sections~\ref{section-some-generalities},
\ref{section-linear-unitary}, \ref{section-symplectic-orthogonal}
and \ref{section-exceptional} we prove various bounds for the number
of $p$-regular and $p'$-regular classes in finite nonabelian simple
groups $S$, as well as the number of $\Aut(S)$-orbits on those
classes. Finally, Theorem~\ref{theorem-p-regular-bound-Lietype} is
proved in Section~\ref{section-Lie-type-bound}.


\section{Orbits of p-regular and p'-regular classes of simple
groups} \label{section-some-generalities}

Let $p$ be a prime and let $S$ be a nonabelian finite simple group.
Recall that an element is $p$-regular in $S$ if it has order coprime
to $p$. We denote the set of $p$-regular elements in $S$ by
$S_{p'}$, the set of $p$-regular conjugacy classes in $S$ by
$\mathrm{Cl}_{p'}(S)$, and the number of $p$-regular conjugacy
classes in $S$ by $k_{p'}(S)$. We denote the set of
\emph{non-trivial} $p$-elements in $S$ by $S_p$, the set of all
conjugacy classes in $S$ contained in $S_p$ by $\mathrm{Cl}_{p}(S)$,
and the number of conjugacy classes in $S$ contained in
$\mathrm{Cl}_{p}(S)$ by $k_{p}(S)$. When a group $G$ acts on a set
$X$, we use $n(G,X)$ to denote the number of $G$-orbits on $X$.

To prove our main results we need to bound the number of
$\Aut(S)$-orbits on $p$-regular and $p'$-regular classes of $S$ for
all nonabelian simple groups $S$, as presented in the following theorem. We
will prove it in this and the next three sections.

\begin{theorem}\label{proposition-simple-groups}
Let $S$ be a nonabelian finite simple group and let $p$ be a prime
divisor of $|S|$. We have

\begin{itemize}

\item[(i)] The number of $\Aut(S)$-orbits on the set $\mathrm{Cl}_{p'}(S)
\cup \mathrm{Cl}_{p}(S)$ is larger than $2\sqrt{p-1}$ except if
$(S,p)$ is equal to $(\Al_5,5)$ or to $(\PSL_{2}(16)),17)$. 

\item[(ii)] The number of $\Aut(S)$-orbits on $p$-regular classes of $S$
is at least $2(p-1)^{1/4}$. The equality occurs if and only if
$(S,p)=(\PSL_{2}(16)),17)$.

\item[(iii)] The number of $\Aut(S)$-orbits on $p$-regular classes of $S$
is greater than $2\sqrt{p-1}$ unless
possibly when $(S,p)$ is listed in Table~\ref{table-exceptions}.
\end{itemize}
\end{theorem}

\begin{table}[ht]
\caption{Possible exceptions for the bound
$n(\Aut(S),\Cl_{p'}(S))>2\sqrt{p-1}$.\label{table-exceptions}}
\begin{center}
\begin{tabular}{llc}
\hline
$S$  & $p$ & $n(\Aut(S),\Cl_{p'}(S))$ \\
\hline

$\Al_5$& $5$ & $3$\\

$\PSL_2(7)$& $7$ & $4$\\

$\Al_6$& $5$ & $4$\\

$\PSL_2(8)$& $7$ & $4$\\


$\PSL_2(11)$& $11$ & $6$\\

$\PSL_2(16)$& $17$ & $5$\\

$\PSL_2(27)$& $13$ & $5$\\

$\PSL_2(32)$& $11$ & $6$\\

$\PSL_2(32)$& $31$ & $6$\\

$\PSL_2(81)$& $41$ & $10$ \\

$\PSL_2(128)$& $43$ & $12$ \\

$\PSL_2(128)$& $127$ & $12$ \\

$\PSL_2(243)$& $61$ & $15$ \\

$\PSL_2(256)$& $257$ & $21$ \\


$\PSL_3(8)$& $73$ & $13$\\

$\PSU_3(16)$& $241$ & $\geq 27$\\


$\ta B_2(8)$& $13$ & 6\\

$\ta B_2(32)$& $31$ & 8\\

$\ta B_2(32)$& $41$ & 9\\

$\ta B_2(128)$& $113$ & $\geq19$\\

$\ta B_2(128)$& $127$ & $\geq14$\\


$\Omega_8^-(4)$& $257$ & $\geq32$\\

\hline
\end{tabular}
\end{center}
\end{table}

\subsection{Some generalities}
Observe that any nonabelian finite simple group has order divisible
by at least three distinct primes by Burnside's Theorem. It
immediately follows that $$n(\Aut(S), \mathrm{Cl}_{p'}(S)) \geq 3$$
and so Theorem~\ref{proposition-simple-groups} is true for $p=2$ and
$p=3$. Thus we may assume in this and the following sections that $p
\geq 5$.

\begin{lemma}\label{lemma-defining-characteristic}
Theorem~\ref{proposition-simple-groups} is true for $S$ a sporadic simple
group, the Tits group, and groups of Lie type of rank $r\geq 3$ in
characteristic $p\geq 5$.
\end{lemma}

\begin{proof}
The statement follows for $S$ a sporadic simple group or $S$ the Tits group using \cite{Atl1,GAP}.
Assume that $S$ is of Lie type of rank $r\geq 3$ in
characteristic $p$. Let $S$ be of the form $G/\bZ(G)$, where
$G=\mathcal{G}^F$ is the set of fixed points of a simple algebraic
group $\mathcal{G}$ of simply connected type defined in characteristic $p$, under
a Frobenius endomorphism $F$. By \cite[Theorem 3.7.6]{Carter}, the
number of semisimple classes of $G$ is $q^r$, where $q$ is the size
of the underlying field of $\mathcal{G}$ and $r$ is the rank of
$\mathcal{G}$. Therefore,
\[
k_{p'}(S)\geq \frac{k_{p'}(G)}{k_{p'}(\bZ(G))}\geq
\frac{q^r}{|\bZ(G)|}=\frac{q^r}{d},
\]
where $d$ is the order of the group of diagonal automorphisms of
$S$. It follows that
\[
n(\Aut(S),\Cl_{p'}(S))\geq \frac{q^r}{d\cdot |\Out(S)|}.
\]
To prove the lemma, it is sufficient to show that $q^r/(d|\Out(S)|)>
2\sqrt{p-1}$. This turns out to be true for all $S$ and relevant
values of $p, q$, and $r$.
\end{proof}

The next result is essential in our proofs as it helps to reduce
from a classical group to one of smaller rank. From now on $q$ is
always a prime power $\ell^{f}$, where $\ell$ is a prime and $f$ is
a positive integer.

Let $\pi$ be either of the symbols $p$ or $p'$. We denote the
number of $\Aut(S)$-orbits on the set $\mathrm{Cl}_{\pi}(S)$ by
$n(\Aut(S), \mathrm{Cl}_{\pi}(S))$.

\begin{lemma}
\label{piPSL} If $S$ and $T$ are (non-abelian) finite simple groups
such that $$(S,T) \in \{ (\Al_{n}, \Al_{n-1}), (\PSL_{n}(q),
\PSL_{n-1}(q)), (\PSU_{n}(q), \PSU_{n-1}(q)) \}$$ or $(S,T) =
(\PSp_{2n}(q), \PSp_{2n-2}(q))$ with $q$ odd, then
$$n(\Aut(S),\mathrm{Cl}_{\pi}(S)) \geq n(\Aut(T),
\mathrm{Cl}_{\pi}(T)).$$
\end{lemma}

\begin{proof}
Let $(S,T) = (\Al_{n}, \Al_{n-1})$. Observe that $n \geq 6$ by
assumption. The group $T$ may be considered as a point-stabilizer in
$S$. Note that $\Aut(\Al_{m}) = \Sy_{m}$ for every integer $m$ at
least $5$ and different from $6$. Assume that $n \geq 8$. In this
case $n(\Aut(S), \mathrm{Cl}_{\pi}(S))$ and $n(\Aut(T),
\mathrm{Cl}_{\pi}(T))$ are equal to the number of elements in
$S_{\pi}$ and $T_{\pi}$ of different cycle shapes. The desired bound
follows since $T_{\pi}$ is contained in $S_{\pi}$. Assume that $n
\in \{ 6, 7 \}$. The group $\Aut(\Al_{6})$ contains $\Sy_{6}$ as a
subgroup of index $2$ and $\Aut(\Al_{6})$ fuses the two conjugacy
classes of $\Al_{6}$ both consisting of elements of order $3$. Since
$p \geq 5$, we see that $n(\Aut(\Al_{6}), \mathrm{Cl}_{p}(\Al_{6}))$
is equal to the number of possible cycle shapes of non-trivial
$p$-elements in $\Al_{6}$ and $n(\Aut(\Al_{6}),
\mathrm{Cl}_{p'}(\Al_{6}))$ is equal to the number of possible cycle
shapes of $p$-regular elements in $\Al_{6}$ minus $1$. The desired
bound follows for $\pi = p$ and also for $n=7$. Let $n = 6$ and $\pi
= p'$. Since $p \geq 5$ and $n = 6$, we must have $p=5$. Finally,
$n(\Aut(\Al_{6}),\mathrm{Cl}_{5'}(\Al_{6})) = 4 > 3 =
n(\Aut(\Al_{5}),\mathrm{Cl}_{5'}(\Al_{5}))$.

Observe that $n \geq 3$ since $T$ is assumed to be non-abelian and
simple.

Consider the case when both $S$ and $T$ are projective special
linear groups.

Let $V$ be the natural $\GL_{n}(q)$-module of dimension $n$ defined
over the field of size $q$. The group $\SL_{n}(q)$ acts naturally on
$V$. Let $W$ be a $1$-dimensional subspace in $V$ and let $U$ be a
complementary $(n-1)$-dimensional subspace in $V$. There is a group
$\SL_{n-1}(q)$ which acts naturally on $U$ and which fixes $W$. All
automorphisms of $S$ can be described by automorphisms of
$\SL_{n}(q)$ and the group $\Out(S)$ is isomorphic to $D_{2(n,q-1)}
\times C_{f}$, see \cite[Section 3.3.4]{Wilson}. The field
automorphisms and the inverse transpose automorphism of $\SL_{n}(q)$
restrict naturally to the subgroup $\SL_{n-1}(q)$ just defined.
Moreover, elements $a$ and $b$ of $\SL_{n-1}(q)$ lie in the same
$\GL_{n}(q)$-orbit if and only if $a$ and $b$ are conjugate in
$\GL_{n-1}(q)$ by a theorem on rational canonical forms.

It follows that the number $N$ of orbits of $\Aut(\SL_{n}(q))$ on
the subset consisting of those elements of
$\mathrm{Cl}_{\pi}(\SL_{n}(q))$ whose members fix some
$1$-dimensional subspace of $V$ is at least $n(\Aut(\SL_{n-1}(q)),
\mathrm{Cl}_{\pi}(\SL_{n-1}(q)))$, which in turn is at least
$n(\Aut(\PSL_{n-1}(q)), \mathrm{Cl}_{\pi}(\PSL_{n-1}(q)))$.

On the other hand, $N$ is at most $n(\Aut(\PSL_{n}(q)),
\mathrm{Cl}_{\pi}(\PSL_{n}(q)))$.

Let $S$ and $T$ be projective special unitary groups.

Let $V$ be the natural $\GU_{n}(q)$-module of dimension $n$ defined
over the field of size $q^{2}$. The module $V$ is equipped with a
non-singular conjugate-symmetric sesquilinear form $f$. The group
$\SU_{n}(q)$ acts naturally on $V$. Let $W$ be a $1$-dimensional
non-singular subspace in $V$ with respect to $f$. Let $U$ be the
$(n-1)$-dimensional non-singular subspace of $V$ perpendicular to
$W$ with respect to the form $f$. There is a subgroup $\SU_{n-1}(q)$
which acts naturally on $U$ and which fixes $W$. The automorphisms
of the simple group $\PSU_{n}(q)$ are described in \cite[Section
3.6.3]{Wilson}. All outer automorphisms of $\PSU_{n}(q)$ come from
outer automorphisms of $\SU_{n}(q)$, these are diagonal
automorphisms or field automorphisms. Field automorphisms preserve
the subgroup $\SU_{n-1}(q)$. By a result of Wall \cite[p. 34, 13,
2]{Wall}, elements $a$ and $b$ of $\SU_{n-1}(q)$ lie in the same
$\GU_{n}(q)$-orbit if and only if $a$ and $b$ are conjugate in
$\GU_{n-1}(q)$.

The proof can now be completed as in the linear case by replacing
the groups $\SL_{n}(q)$, $\SL_{n-1}(q)$, $\PSL_{n-1}(q)$,
$\PSL_{n}(q)$ by $\SU_{n}(q)$, $\SU_{n-1}(q)$, $\PSU_{n-1}(q)$,
$\PSU_{n}(q)$ respectively.

Let $S$ and $T$ be projective symplectic groups.

Let $V$ be the natural $\Sp_{2n}(q)$-module of dimension $n$ defined
over the field of size $q$. The module $V$ is equipped with a
non-singular alternating bilinear form $f$. Let $W$ be a
$2$-dimensional non-singular subspace in $V$ with respect to $f$.
Let $U$ be the $(2n-2)$-dimensional non-singular subspace of $V$
perpendicular to $W$ with respect to the form $f$. There is a
subgroup $\Sp_{2n-2}(q)$ which acts naturally on $U$ and which fixes
$W$. The automorphisms of the simple group $\PSp_{2n}(q)$ are
described in \cite[Section 3.5.5]{Wilson}. All outer automorphisms
of $\PSp_{2n}(q)$ come from outer automorphisms of $\Sp_{2n}(q)$ and
these are diagonal automorphisms and field automorphisms since we
are assuming that $q$ is odd. All outer automorphisms of
$\Sp_{2n}(q)$ preserve the subgroup $\Sp_{2n-2}(q)$. By a result of
Wall \cite[p. 36]{Wall}, elements $a$ and $b$ of $\Sp_{2n-2}(q)$ lie
in the same $\Sp_{2n}(q)$-orbit if and only if $a$ and $b$ are
conjugate in $\Sp_{2n-2}(q)$. See also \cite[210, 211]{Fulman1}.

The proof can now be completed as in the linear or unitary case.
\end{proof}

To have good estimates of $n(\Aut(S),\Cl_{p'}(S))$ and
$n(\Aut(S),\Cl_{p}(S)\cup \Cl_{p'}(S))$, especially for low rank
classical groups and exceptional groups, we will use so-called
strongly self-centralizing maximal tori. A subgroup $T$ of $G$ is
said to be \emph{strongly self-centralizing} if $\bC_G(t)=T$ for
every $1\neq t\in T$. These groups are useful because of the
following lemma due to Babai, P\'alfy and Saxl.

\begin{lemma}\label{lemma-strongly-self-centralizing}
Let $G$ be a finite group with a strongly self-centralizing subgroup
$T$. Let $p$ be a prime.
\begin{itemize}
\item[(i)] If $p\mid |T|$, then \[|G_{p'}|>
|G|-\dfrac{|G|}{|\bN_G(T)/T|}\] and \[|G_{p}|\geq
\frac{|G|}{|\bN_G(T)/T|}\frac{|T|-1}{|T|}>
\frac{|G|}{|\bN_G(T)/T|+1}.\]

\item[(ii)] If $p\nmid |T|$, then
\[|G_{p'}|\geq\frac{|G|}{|\bN_G(T)/T|}\frac{|T|-1}{|T|}>
\dfrac{|G|}{|\bN_G(T)/T|+1}.\] Moreover, if $G$ contains pairwise
non-conjugate strongly self-centralizing subgroups $T_1$, $T_2$,
\ldots , $T_k$ such that $p\nmid |T_i|$ for all $1\leq i\leq k$,
then
\[
|G_{p'}|> \sum_{i=1}^k\dfrac{|G|}{|\bN_G(T_i)/T_i|+1}.
\]
\end{itemize}
\end{lemma}

\begin{proof}
See \cite[Proposition 1.15]{Babai-Palfy-Saxl09} and its proof. There
the authors used the language of proportion of $p$-regular elements
but one can transfer to the number of $p$-regular elements.
\end{proof}

\subsection{Alternating groups} We finish this section by proving
Theorem~\ref{proposition-simple-groups} for the alternating groups.

\begin{lemma}
\label{LemAlt} Theorem \ref{proposition-simple-groups} holds for $S
= \Al_{n}$ with $n \geq 5$.
\end{lemma}

\begin{proof}
Assume first that $p \geq 7$. We claim that $n(\Aut(S),\Cl_{p'}(S))
> 2 \sqrt{p-1}$. Assume for a contradiction that $S = \Al_{n}$ is a
counterexample to the claim with $n \geq 7$ minimal. The prime $p$
divides $|\Al_{n}|$ but does not divide $|\Al_{n-1}|$ by Lemma
\ref{piPSL}. It follows that $n = p$. The group $\Al_p$ has cycles
of every odd length up to $p-2$ and has $\lfloor p/3 \rfloor$ cycle
types of elements of order $3$. Therefore we have \[
n(\Aut(S),\Cl_{p'}(S))\geq (p-1)/2 + \lfloor p/3 \rfloor > 2
\sqrt{p-1},\] from which the claim follows.

Let $p = 5$. From the previous paragraph we have
$n(\Aut(\Al_{7}),\Cl_{p'}(\Al_{7})) > 2 \sqrt{p-1}$. This implies
$n(\Aut(S),\Cl_{p'}(S)) > 2 \sqrt{p-1}$ for $n \geq 7$ by Lemma
\ref{piPSL}. We find $n(\Aut(\Al_{6}),\mathrm{Cl}_{5'}(\Al_{6})) =
4$ and $n(\Aut(\Al_{5}),\mathrm{Cl}_{5'}(\Al_{5})) = 3$. Parts (ii)
and (iii) follow.

The number of orbits of $\Aut(S)$ on $\mathrm{Cl}_{p}(S) \cup
\mathrm{Cl}_{p'}(S)$ is $4$ if $S = \Al_{5}$ and is $5$ if $S =
\Al_{6}$. Part (i) follows and the proof is complete.
\end{proof}



\section{Theorem \ref{proposition-simple-groups}: Linear and unitary
groups}\label{section-linear-unitary}

\subsection{Linear groups}\label{subsection-linear} In this subsection, we will prove
Theorem~\ref{proposition-simple-groups}(i) for $S=\PSL_n(q)$ with
$n\geq 2$, $q=\ell^f$ where $\ell$ is a prime, and $(n,q)\notin
\{(2,2), (2,3)\}$. We keep the notation introduced in
Section~\ref{section-some-generalities} and start with the following
technical lemma.

\begin{lemma}
\label{primitiveprimedivisor} Let $n \geq 2$ be an integer and let
$\epsilon$ be $1$ or $2$ depending on whether $n = 2$ or $n > 2$
respectively. Let $S = \PSL_{n}(q)$ be a simple group. Let $$m =
\frac{q^{n}-1}{(q-1)(n,q-1)}.$$ If $p$ divides $m$, then
$$n(\Aut(S),\Cl_{p}(S)) \geq \frac{p-1}{\epsilon f n}.$$ If $p$ does
not divide $m$, then $$n(\Aut(S),\Cl_{p'}(S)) \geq
\frac{\varphi(m)}{\epsilon f n}$$ where $\varphi$ is Euler's totient
function.
\end{lemma}

\begin{proof}
Let $A = \Aut(S)$ and let $a$ be an element of $S$. Let $g$ be the
preimage of $a$ in $\SL_{n}(q) \leq \GL_{n}(q)$. Assume that $g$
acts irreducibly on $V$. The centralizer of $g$ in $\GL_{n}(q)$ is a
cyclic group $C$ of order $q^{n}-1$ and the normalizer of $\langle g
\rangle$ in $\GL_{n}(q)$ is $C:\langle \sigma \rangle$ where
$\sigma$ is a field automorphism of order $n$. There is a subgroup
$B$ in $A$ defined in a natural way which contains $S$ and which is
isomorphic to $\PGL_{n}(q)$. We have $|B:S| = (n,q-1)$ and $|A:B| =
\epsilon f$. Observe that $$|\bC_{B}(a)| = \frac{|C|}{(n,q-1)}$$ and
$$|\bN_{B}(\langle a \rangle)| = \frac{n|C|}{(n,q-1)}.$$ It follows
that $|\bC_{A}(a)| \geq |\bC_{B}(a)| = |C|/(n,q-1)$ and
$$|\bN_{A}(\langle a \rangle)| \leq \epsilon f |\bN_{B}(\langle a
\rangle)| \leq \frac{\epsilon fn|C|}{(n,q-1)}.$$ Thus
$|\bN_{A}(\langle a \rangle)/\bC_{A}(\langle a \rangle)| \leq
\epsilon fn$. Assume now that $a$ is of order $m$. It follows that
there are at least $\varphi(m)/(\epsilon fn)$ conjugacy classes of
$A$ all contained in $S$ which consist of elements of order $m$. The
desired bound now follows in the case when $p$ does not divide $m$.
If $p$ divides $m$, then the bound also follows by noting that
$\varphi(m) \geq p-1$.
\end{proof}

\begin{lemma}\label{PSL2}
Theorem \ref{proposition-simple-groups} holds for $S = \PSL_{2}(q)$
with $q \geq 4$.
\end{lemma}

\begin{proof}
Let $q \leq 256$. By a Gap \cite{GAP} calculation
$n(\Aut(S),\Cl_{p'}(S)) > 2 \sqrt{p-1}$ unless $q \in \{ 4, 5, 7, 8,
9, 11, 16, 27, 32, 81, 128, 243, 256 \}$. The exceptional cases
account for the possibilities in Table \ref{table-exceptions}.
Furthermore, if $q$ belongs to $\{ 512, 1024 \}$, then
$n(\Aut(S),\Cl_{p'}(S)) > 2 \sqrt{p-1}$ for every possible value of
$p$.

Assume first that $p$ divides $q+1$. There are $q-1$ diagonal
elements in $\SL_{2}(q)$ with respect to a fixed basis. Thus there
are at least $q-1$ conjugacy classes of $p$-regular elements in
$\GL_{2}(q)$ and so at least $(q-1)/2$ conjugacy classes of
$\PGL_{2}(q)$ consisting of $p$-regular elements in $\PSL_{2}(q)$.
Thus $n(\Aut(S),\Cl_{p'}(S)) \geq (q-1)/(2f)$. This is larger than
$2 \sqrt{q}$ subject to the restrictions $q > 256$ and $q \not\in \{
512, 1024 \}$. This proves parts (ii) and (iii) in the case when $p$
divides $q+1$. We now turn to the proof of part (i) in the case when
$p$ divides $q+1$. We may assume that the pair $(S,p)$ appears in
Table~\ref{table-exceptions}.

We have $n(\Aut(S),\Cl_{p}(S)) \geq (p-1)/(2 f)$ by Lemma
\ref{primitiveprimedivisor}. The exact values of
$n(\Aut(S),\Cl_{p'}(S))$ may be found in Table
\ref{table-exceptions}. Using this information, for any pair $(S,p)$
in Table \ref{table-exceptions} such that $p$ divides $q+1$, we get
$$n(\Aut(S),\Cl_{p'}(S)) + n(\Aut(S),\Cl_{p}(S)) > 2 \sqrt{p-1},$$
unless $(S,p) = (\PSL_{2}(16),17)$ when $n(\Aut(S),\Cl_{p'}(S)) +
n(\Aut(S),\Cl_{p}(S)) = 7$. This latter pair is an exception in part
(i). This proves part (i) in the case $p$ divides $q+1$.

Assume now that $p$ does not divide $q+1$. In this case $p = \ell$
or $p \mid q-1$.

Let $p = \ell$. There are at least $(\ell-1)/2$ diagonal elements in
$\SL_{2}(\ell)$. These elements are fixed by field automorphisms and
no two of them are conjugate in $\GL_{2}(q)$. Thus
$n(\Aut(S),\Cl_{p'}(S)) \geq (\ell-1)/2$. In fact, by the proof of
Lemma \ref{primitiveprimedivisor}, $$n(\Aut(S),\Cl_{p'}(S)) \geq
\varphi \Big(\frac{q+1}{(2,q-1)}\Big)/(2f) + (\ell-1)/2.$$ This is
larger than $2 \sqrt{\ell - 1}$ unless $\ell = q \in \{ 5, 7, 11
\}$. The pairs $(S,p)$ in $$\{ (\PSL_{2}(5),5), (\PSL_{2}(7),7),
(\PSL_{2}(11),11) \}$$ appear in Table \ref{table-exceptions}. This
proves parts (ii) and (iii) in the case $p = \ell$. Part (i) in the
case $p = \ell$ also follows by a direct check using \cite{Atl1}.

Finally, assume that $p$ divides $q-1$. We may assume by the first
paragraph of this proof that $q > 256$ and $q \not\in \{ 512, 1024
\}$. We have $$n(\Aut(S),\Cl_{p'}(S)) \geq \varphi
\Big(\frac{q+1}{(2,q-1)}\Big)/(2f)$$ by Lemma
\ref{primitiveprimedivisor}. This is larger than $2 \sqrt{p-1}$
unless $q \in \{ 263, 359 \}$. Another \cite{GAP} calculation gives
$n(\Aut(S),\Cl_{p'}(S)) > 2 \sqrt{p-1}$ for $S \in \{ \PSL_{2}(263),
\PSL_{2}(359) \}$.
\end{proof}


\begin{lemma}
\label{PSL3} Theorem \ref{proposition-simple-groups} holds for $S =
\PSL_{3}(q)$.
\end{lemma}

\begin{proof}
We will show that $n(\Aut(S),\Cl_{p'}(S)) > 2 \sqrt{p-1}$
in all cases except when $(S,p) = (\PSL_{3}(8),73)$.

We may exclude $q \leq 9$ using \cite{Atl1} and $q \in \{ 13, 16, 19
\}$ using \cite{GAP}.

From \cite{Simpson73} we observe that $|\bC_S(g)|\geq q^2/(3,q-1)$
for every $g\in S$. We also know that $S$ has a strongly
self-centralizing maximal torus $T$ of order
$$(q^2+q+1)/(3,q-1)=\Phi_3(q)/(3,q-1)$$ with $|\bN_S(T)/T|=3$, see
\cite[p. 16]{Babai-Palfy-Saxl09} for instance.

Suppose first that $p\mid |T|$. Then by Lemma
\ref{lemma-strongly-self-centralizing}, we have $|S_{p'}|
>\frac{2}{3}|S|$ and thus
\[
k_{p'}(S)>\frac{2q^2}{3(3,q-1)},
\]
which yields
\[n(\Aut(S),\Cl_{p'}(S))>\frac{q^2}{3f(3,q-1)^2}=:R(q).\]
One can check that $R(q)\geq 2\sqrt{\Phi_3(q)/(3,q-1)-1}$ unless $q
\in \{ 25, 49, 64 \}$ (as we already excluded the case $q \in \{ 13,
16, 19 \}$). Checking further, we find that the desired inequality
$R(q)\geq 2\sqrt{p-1}$ still holds.

Now we suppose $p\nmid |T|$. By Lemma
\ref{lemma-strongly-self-centralizing}, we have $|S_{p'}|
>|S|(|T|-1)/3|T|$. Hence
\[k_{p'}(S)>\frac{q^2(|T|-1)}{3(3,q-1)|T|},\] implying that
\[n(\Aut(S),\Cl_{p'}(S))>\frac{q^2(|T|-1)}{6f(3,q-1)^2|T|}:=R'(q).\]
It is easy to check that $R'(q)\geq 2\sqrt{q}\geq 2\sqrt{p-1}$
unless $q \in \{ 25, 64 \}$ (as we excluded the case $q \in \{ 13,
16 \}$). In fact we still have $R'(q)\geq 2\sqrt{p-1}$ when $q \in
\{ 25, 64 \}$ since $p\leq 13$ in those cases. This proves parts
(ii) and (iii) by noting that the pair $(S,p) = (\PSL_{3}(8),73)$
appears in Table \ref{table-exceptions}.

For the proof of part (i) we may now assume that $(S,p) =
(\PSL_{3}(8),73)$. Then $S$ has a maximal torus of order
$\Phi_3(8)=73$ with the relative Weyl group of order $3$, and thus
$S$ has at least $72/3=24$ conjugacy classes of elements of order
73. It follows that there are at least $24/3=8$ $\Aut(S)$-orbits on
$\Cl_p(S)$. We now have $n(\Aut(S),\Cl_p(S)\cup \Cl_{p'}(S))\geq
8+13>2\sqrt{p-1}$, as desired.
\end{proof}

\begin{lemma}
Theorem \ref{proposition-simple-groups}(i) holds for $S =
\PSL_{n}(q)$ with $n \geq 4$.
\end{lemma}

\begin{proof}
Assume for a contradiction that part (i) fails for the group $S =
\PSL_{n}(q)$ with $n \geq 4$ minimal.

The prime $p$ divides $|\PSL_{n}(q)|$ but does not divide
$|\PSL_{n-1}(q)|$ by Lemma \ref{piPSL}. This implies that $p$
divides $$m = \frac{q^{n}-1}{(q-1)(n,q-1)}.$$ We get
$$n(\Aut(S),\Cl_{p}(S)) \geq \frac{p-1}{2 f n}$$ by Lemma
\ref{primitiveprimedivisor}. By Lemma \ref{piPSL} and
\cite[Corollary 3.7 (2)]{Fulman-Guralnick12} we also have
$$n(\Aut(S), \mathrm{Cl}_{p'}(S)) \geq \frac{k(\PSL_{n-1}(q))}{2
f(n-1,q-1)} \geq \frac{q^{n-2}}{2 f {(n-1,q-1)}^{2}}.$$ From these
it follows that
$$n(\Aut(S),\Cl_{p}(S)) + n(\Aut(S), \mathrm{Cl}_{p'}(S)) \geq$$
$$\geq \frac{p-1}{2 f n} + 2fn + \frac{q^{n-2}}{\delta f {(n-1,q-1)}^{2}} -
2fn \geq 2 \sqrt{p-1} + \frac{q^{n-2}}{2 f {(n-1,q-1)}^{2}} - 2fn.$$
We may thus assume that $q^{n-2} \leq 4 f^{2}n {(n-1,q-1)}^{2}$.

An easy calculation gives $n \leq 9$ and $q < 128$. Moreover, for $n
\leq 9$ and $q < 128$, a Gap \cite{GAP} computation yields that
$(n,q)$ must belong in $$\{ (4,2), (4,3), (4,4), (4,7), (4,8),
(4,16), (4,64), (5,2), (5,3), (5,4), (5,5), (5,9), (6,2) \}.$$

Since $p$ divides $m$ but $p \nmid |\PSL_{n-1}(q)|$, the triple
$(n,q,p)$ must be $(4,2,5)$, $(4,3,5)$, $(4,4,17)$, $(4,7,5)$,
$(4,8,5)$, $(5,8,13)$, $(4,16,257)$, $(4,64,17)$, $(4,64,241)$,
$(5,2,31)$, $(5,3,11)$, $(5,4,11)$, $(5,4,31)$, $(5,5,11)$,
$(5,5,71)$, $(5,9,11)$ or $(5,9,61)$. In all these cases, except
when $(n,q,p) = (4,16,257)$, we find that $n(\Aut(S),
\mathrm{Cl}_{p'}(S)) > 2 \sqrt{p-1}$, using Gap \cite{GAP}
calculations combined with Lemma \ref{piPSL} together with the bound
$n(\Aut(S), \mathrm{Cl}_{p'}(S)) \geq k(\PSL_{n-1}(q))/2f(n-1,q-1)$.

Let $(n,q,p) = (4,16,257)$. The number of conjugacy classes of $S =
\PSL_{4}(16)$ is $4368$ by \cite{GAP}. Observe that $m = 17 \cdot
257$ in this case. Every element $a$ in $S$ which is not $p$-regular
and not a $p$-element has order $m$. There are $\varphi(m)/2 = 2048$
possible conjugacy classes of such elements $a$ in $S$. It follows
that
$$n(\Aut(S),\Cl_{p}(S)) + n(\Aut(S), \mathrm{Cl}_{p'}(S))
\geq (4368 - 2048)/2f(n-1,q-1) > 2 \sqrt{p-1},$$ and the proof is
complete.
\end{proof}

\subsection{Unitary groups}\label{subsection-unitary} We continue to prove
Theorem~\ref{proposition-simple-groups}(i) for $S=\PSU_n(q)$.

\begin{lemma}
\label{PSU3} Theorem \ref{proposition-simple-groups} holds for $S =
\PSU_3(q)$.
\end{lemma}

\begin{proof}
We show that $n(\Aut(S),\Cl_{p'}(S)) > 2\sqrt{p-1}$ unless
$(S,p) = (\PSU_{3}(16),241)$. This exceptional case appears in Table
\ref{table-exceptions}.

We follow the same idea as in the proof of Lemma~\ref{PSL3}. The
case $q \leq 11$ can be checked directly using \cite{Atl1}. Thus
assume that $q \geq 13$.

%
%

For every $g \in G:=\PGU_3(q)$ we have $|\bC_G(g)|\geq q^2-q+1$ by
\cite{Simpson73}. Therefore the centralizer size in $\Aut(S)$ of an
element in $S$ is at least $q^2-q+1$, which in turn implies that
the size of every $\Aut(S)$-orbit on $S$ is at most
$|\Aut(S)|/(q^2-q+1)$.

On the other hand, $S$ has a strongly self-centralizing maximal
torus $T$ of order
$$(q^2-q+1)/(3,q+1)=\Phi_6(q)/(3,q+1)$$ with $|\bN_S(T)/T|=3$.
Suppose first that $p\mid |T|$. Using Lemma
\ref{lemma-strongly-self-centralizing}, we obtain $|S_{p'}|>2|S|/3$
and, together with the conclusion of the previous paragraph, we
deduce that
\[n(\Aut(S),\Cl_{p'}(S))>\frac{(q^2-q+1)}{3f(3,q+1)}.\]
One can check that
$(q^2-q+1)/(3f(3,q+1))>2\sqrt{(q^2-q+1)/(3,q+1)-1}$ unless $q =16$
(since we are assuming $q \geq 13$).

For $q=16$ we have $p=241$, $n(\Aut(S),\Cl_{p'}(S))>20$, and
$$n(\Aut(S),\Cl_{p}(S))> (q^2-q+1)(|T|-1)/(24|T|)=10,$$ which implies
that $n(\Aut(S),\Cl_{p'}(S)\cup \Cl_{p}(S))>31>2\sqrt{p-1}$, as
wanted. In fact, we are going to show that
$n(\Aut(S),\Cl_{p}(S))\geq 27$ for $(S,p)=(\PSU_3(16),241)$, as
appeared in Table~\ref{table-exceptions}.

The conjugacy classes and the character table of $\SU_3(q)$ as well as of
$\PSU_3(q)$ are available in \cite{Simpson73}. We observe that $S$
has three unipotent classes, each of which is invariant under
$\Aut(S)$. We will see that $S$ has at least 24 $\Aut(S)$-orbits on
semisimple elements of order coprime to $p=241$. First $S$ has 16
classes labeled by $C_4^{(k)}$ for $1\leq k\leq 16$ of elements of
order 17 and these classes produce at least 2 $\Aut(S)$-orbits. Next
there are 40 classes labeled by $C_6^{(k,l,m)}$ for $1\leq k,l,m\leq
17$, $k<l<m$, and $k+l+m \equiv 0 \pmod {17}$ of elements of order 17,
which produces at least $5$ other orbits. Finally there are 119
classes labeled by $C_7^{(k)}$ for $1\leq k< 15\cdot 17$ and
$k\not\equiv 0 \pmod {15}$ elements of orders dividing $15\cdot
17$. Among those 119 classes there is a single class of elements of
order 3 and hence that class constitutes a single orbit. The other
118 classes make up at least 16 orbits, with the notice that the
size of each orbit is a divisor of $|\Out(S)|=8$.

Now we suppose $p\nmid |T|$ and thus $p\leq q+1$. Again by Lemma
\ref{lemma-strongly-self-centralizing}, we have $|S_{p'}|
>|S|(|T|-1)/(3|T|)$. Hence
\[k_{p'}(S)>\frac{(q^2-q+1)(|T|-1)}{3(3,q+1)|T|}=\frac{|T|-1}{3},\] implying that
\[n(\Aut(S),\Cl_{p'}(S))>\frac{(|T|-1)}{6f(3,q+1)}=\frac{q^2-q+1-(3,q+1)}{6f(3,q+1)^2}:=R(q).\]
We have $R(q)\geq 2\sqrt{q}\geq 2\sqrt{p-1}$ unless $q \in \{ 16,
17, 23 \}$ (assuming that $q \geq 13$). For $q \in \{16, 17, 23 \}$
the bound $R(q)>2\sqrt{p-1}$ still holds.
\end{proof}

We proceed to prove part (i) for the groups $S = \PSU_{n}(q)$ with
$n \geq 4$. If $(n,q) = (4,2)$, then $p = 5$ since we are assuming
$p \geq 5$ and so $n(\Aut(S), \mathrm{Cl}_{p'}(S)) = 14 > 4$ by
\cite{GAP}. Assume from now on that $(n,q) \not= (4,2)$ (and $n \geq
4$). In this case $\PSU_{n-1}(q)$ is a simple group.

\begin{lemma}
\label{primitiveprimedivisorPSU} In order to prove Theorem
\ref{proposition-simple-groups}(i) for $S = \PSU_{n}(q)$ with $n
\geq 4$, we may assume that
$$n(\Aut(S), \mathrm{Cl}_{p'}(S)) \geq \frac{q^{n-2}}{2f {(n-1,q+1)}^{2}}$$
and that the prime $p$ divides $q^{n}- {(-1)}^{n}$. Moreover, if $n$
is odd then $p$ is a primitive prime divisor of $q^{2n}-1$.
\end{lemma}

\begin{proof}
The prime $p$ divides $|S|$ by assumption and $p \geq 5$. Since
$(n,q) \not= (4,2)$, we may assume by Lemma \ref{piPSL} that $p$
does not divide $|\PSU_{n-1}(q)|$. This has two implications.
Firstly,
$$n(\Aut(S), \mathrm{Cl}_{p'}(S)) \geq \frac{k(\PSU_{n-1}(q))}{2f (n-1,q+1)}
\geq \frac{q^{n-2}}{2f {(n-1,q+1)}^{2}}$$ by Lemma \ref{piPSL} and
\cite[Corollary 3.11 (2)]{Fulman-Guralnick12} and secondly $p \mid
q^{n} - {(-1)}^{n}$.

Let $n \geq 5$ be odd. We claim that $p$ is a primitive prime
divisor of $q^{2n}-1$. Assume for a contradiction that $p \mid
q^{k}-1$ for some integer $k$ with $1 \leq k < 2n$. Since $p$ does
not divide $|\PSU_{n-1}(q)|$, it cannot divide $q^{r}+1$ for $r<n$
odd and it cannot divide $q^{r}-1$ for $r<n$ even.

Assume that $k \leq n$. By the previous paragraph, $k$ must be odd.
Since $p$ divides $(q^{n}+1) + (q^{k}-1)$, the prime $p$ must divide
$q^{n-k}+1$. Thus $k < n$. Since $p$ divides $(q^{k}-1) +
(q^{n-k}+1)$, it must divide $q^{|n-2k|}+1$. This is a contradiction
since $|n-2k|$ is odd. It follows that $n < k$. Since $p$ divides
$(q^{2n}-1) - (q^{k}-1)$, it must divide $q^{2n-k}-1$. This is a
contradiction since $2n-k < n$.
\end{proof}

\begin{lemma}
\label{PSUneven} Theorem \ref{proposition-simple-groups}(i) holds
for $S = \PSU_n(q)$ with $n \geq 4$ even.
\end{lemma}

\begin{proof}
Let $n$ be even. In this case $p$ divides $q^{n}-1$ and thus $p-1
\leq q^{n/2}$ by Lemma \ref{primitiveprimedivisorPSU}. Again by
Lemma \ref{primitiveprimedivisorPSU} we are finished if $q^{n-2}/(2
f {(n-1,q+1)}^{2}) > 2 q^{n/4}$. We may assume that $q^{(3n/4)-2}
\leq 4 f {(n-1,q+1)}^{2}$. It follows that $(n,q)$ belongs to the
set $$\{ (4,3), (4,4), (4,5), (4,8), (4,11), (4,16), (4,17), (4,23),
(4,29), (4,32), (4,128), (6,4) \}.$$ Taking into account that $p \geq
5$ divides $q^{n}-1$ but $p$ does not divide $|\PSU_{n-1}(q)|$ by
Lemma \ref{piPSL}, the triple $(n,q,p)$ must be $(4,3,5)$,
$(4,4,17)$, $(4,5,13)$, $(4,8,5)$, $(4,8,13)$, $(4,11,61)$,
$(4,16,257)$, $(4,17,5)$, $(4,17,29)$, $(4,23,5)$, $(4,23,53)$,
$(4,29,421)$, $(4,32,5)$, $(4,32,41)$, $(4,128,5)$, $(4,128,29)$,
$(4,128,113)$, $(4,128,127)$ or $(6,4,7)$. We get $$\frac{q^{n-2}}{2 f
{(n-1,q+1)}^{2}} > 2 \sqrt{p-1}$$ unless $(n,q,p)$ is equal to
$(4,4,17)$, $(4,5,13)$, $(4,8,5)$, $(4,8,13)$, $(4,11,61)$,
$(4,16,257)$, $(4,32,41)$ or $(6,4,7)$. A Gap \cite{GAP} computation
using Lemma \ref{piPSL} gives $$n(\Aut(S), \mathrm{Cl}_{p'}(S)) > 2
\sqrt{p-1}$$ unless $(n,q,p) = (4,32,41)$. If $(n,q,p) = (4,32,41)$,
then the bound still holds since $331$ is a prime factor of
$|\PSU_{3}(q)|$, unlike $p = 41$, and $$n(\Aut(S),
\mathrm{Cl}_{p'}(S)) \geq n(\Aut(\PSU_{3}(q)),
\mathrm{Cl}_{p'}(\PSU_{3}(q))) \geq$$ $$\geq n(\Aut(\PSU_{3}(q)),
\mathrm{Cl}_{331'}(\PSU_{3}(q))) > 2 \sqrt{331-1} > 2 \sqrt{p-1}$$
by Lemma \ref{piPSL} and the proof of Lemma \ref{PSU3}.
\end{proof}

\begin{lemma}
\label{primitiveprimedivisorunitary} Let $n \geq 5$ be odd. Let $p
\geq 5$ be a prime which divides $q^{n}+1$ and which is a primitive
prime divisor of $q^{2n}-1$. Let $S = \PSU_{n}(q)$. The number of
orbits of $\Aut(S)$ on the set of elements of $S$ of orders
divisible by $p$ but not equal to $p$ is $$\frac{
\frac{q^{n}+1}{(q+1)(n,q+1)} - p }{2fn}.$$
\end{lemma}

\begin{proof}
Let $g \in \GU_{n}(q)$ be an element of order divisible by $p$.
Since $\GU_{n}(q) \leq \GL_{n}(q^{2})$ and $p$ is a primitive prime
divisor of $q^{2n}-1$, we see that $g$ acts irreducibly on the
underlying vector space of dimension $n$ over the field of size
$q^{2}$. It is contained in a Singer cycle $C$ of $\GU_{n}(q)$
defined to be a cyclic irreducible subgroup of $\GU_{n}(q)$ of
maximal possible order and whose existence is proved by Huppert in
\cite{Huppert1}. The group $C$ is the intersection of $\GU_{n}(q)$
with the Singer cycle of $\GL_{n}(q^{2})$ (which is a cyclic
subgroup) containing $g$. Since the centralizer of $g$ in
$\GL_{n}(q^{2})$ is the Singer cycle containing $g$, it follows that
the centralizer of $g$ in $\GU_{n}(q)$ is $C$. The group $C$ has
order $q^{n}+1$ by \cite[Satz 4]{Huppert1}. Since $p$ is a primitive
prime divisor of $q^{2n}-1$, $C$ contains a Sylow $p$-subgroup $P$
of $\GU_{n}(q)$. The centralizer in $\GU_{n}(q)$ of any non-trivial
element of $P$ is $C$. It follows that all Singer cycles in
$\GU_{n}(q)$ are conjugate and also that the centralizer of $g$ in
$\GU_{n}(q)$ is $C$. The group $\bN_{\GL_{n}(q^{2})}(\langle g
\rangle)/\bC_{\GL_{n}(q^{2})}(\langle g \rangle)$ is cyclic of order
$n$, so $\bN_{\GU_{n}(q)}(\langle g \rangle) = C.m$ for some divisor
$m$ of $n$. Since $g$ is contained in an extension field subgroup
$\GU_{1}(q^{n}).n$ of $\GU_{n}(q)$, we obtain $m = n$.

The image $F$ of $C \cap \SU_{n}(q)$ in $S$ has order
$(q^{n}+1)/((q+1)(n,q+1))$. Cyclic subgroups of this order are all
conjugate in $\Aut(S)$ by the previous paragraph. Every element of
$S$ of order divisible by $p$ is contained in some conjugate of $F$
in $\Aut(S)$. Observe that $|\bN_{\Aut(S)}(F)/F| = 2fn$. The lemma
follows.
\end{proof}

We are now in position to complete the proof for the unitary case.

\begin{lemma}
\label{PSUnodd} Theorem \ref{proposition-simple-groups}(i) holds for
$S = \PSU_n(q)$ with $n \geq 5$ odd.
\end{lemma}

\begin{proof}
Let $n \geq 5$ be odd. We may assume that $p$ is a primitive prime
divisor of $q^{2n}-1$ by Lemma \ref{primitiveprimedivisorPSU}. Thus
$$n(\Aut(S), \mathrm{Cl}_{p'}(S)) + n(\Aut(S), \mathrm{Cl}_{p}(S))
\geq \frac{k(S)}{2f (n,q+1)} - \frac{q^{n}+1}{2fn(q+1)(n,q+1)}+
\frac{p}{2fn}$$ by Lemma \ref{primitiveprimedivisorunitary}. This is
at least
$$\frac{p}{2fn} + 2fn - 2fn + \frac{q^{n-1}}{2f {(n,q+1)}^{2}} -
\frac{q^{n}+1}{2fn(q+1)(n,q+1)} \geq$$ $$\geq 2\sqrt{p} +
\frac{q^{n-1}}{2f {(n,q+1)}^{2}} - \frac{q^{n}+1}{2fn(q+1)(n,q+1)} -
2fn$$ by \cite[Corollary 3.11 (2)]{Fulman-Guralnick12}. This latter
expression is at least $2 \sqrt{p}$ if and only if $$q^{n-1} \geq
\frac{(q^{n}+1)(n,q+1)}{n(q+1)} + 4f^{2}n {(n,q+1)}^{2}.$$ This is
satisfied unless $(n,q) \in \{ (5,2), (5,4), (5,9), (9,2) \}$.
Taking into account that the prime $p \geq 5$ divides $q^{n}+1$, the
triple $(n,q,p)$ must belong to $$\{ (5,2,11), (5,4,5), (5,4,41),
(5,9,5), (5,9,1181), (9,2,19) \}.$$ Among these exceptions, we have
$$\frac{q^{n-1}}{2f {(n,q+1)}^{2}} - \frac{q^{n}+1}{2fn(q+1)(n,q+1)} >
2 \sqrt{p-1}$$ unless $(n,q,p) \in \{ (5,4,5), (5,4,41),
(5,9,1181)\}$. If $(n,q,p) \in \{ (5,4,5), (5,4,41) \}$, then
$$n(\Aut(S), \mathrm{Cl}_{p'}(S)) \geq n(\Aut(\PSU_{4}(q)),
\mathrm{Cl}_{p'}(\PSU_{4}(q))) > 2 \sqrt{p-1}$$ by Lemma \ref{piPSL}
and \cite{GAP}. Let $(n,q,p) = (5,9,1181)$. The precise number of
conjugacy classes of $S$ can be computed using \cite{Macdonald}.
This is $k(S) = 7596$. Plugging this into the first displayed
expression of the present proof, we obtain $$n(\Aut(S),
\mathrm{Cl}_{1181'}(S)) + n(\Aut(S), \mathrm{Cl}_{1181}(S)) > 2
\sqrt{1180},$$ and this finishes the proof.
\end{proof}


\subsection{Theorem \ref{proposition-simple-groups}(ii) and (iii): Linear and
unitary groups of dimension at least 4}

The method in Subsections
\ref{subsection-linear} and \ref{subsection-unitary} can be revised
to prove parts (ii) and (iii) for $S=\PSL_n(q)$ and $\PSU_n(q)$, but
we present here another path to do it. As the case $n\leq 3$ has
been proved in Lemmas~\ref{PSL2}, \ref{PSL3}, and \ref{PSU3}, we
will assume that $n\geq 4$ in this subsection.

We use $\PSL^+_n(q)$ for the linear groups and $\PSL^-_n(q)$ for the
unitary groups.

\begin{lemma}\label{lemma-linear-unitary-bound}
Let $S=\PSL^\epsilon_n(q)$ for $n\geq 4$ and $p$ a prime divisor of
$|S|$ but $p\nmid q$. Assume that $p\mid (q^n-(\epsilon1)^n)$ but
$p\nmid (q^i-(\epsilon 1)^i)$ for every $1\leq i\leq n-1$. Then

\[n(\Aut(S),\Cl_{p'}(S))
> \frac{q^{n-1}(n-1)}{2nf(n,q-\epsilon1)} H(n,q,\epsilon),\]
where \[H(n,q,+)=\frac{1}{er}\] with $r:=\min\{x \in\NN: x\geq
\log_q(n+1)\}$ and
\[H(n,q,-)=\left(\frac{q^2-1}{er'(q+1)^2}\right)^{1/2}\] with
$r':=\min\{x\in \NN: x \text{ odd and } x\geq \log_q(n+1)\}$.
\end{lemma}

\begin{proof}
By \cite[Theorems 6.4 and 6.7]{Fulman-Guralnick12} and their proofs,
the minimal centralizer size of an element in $\GL^\epsilon_{n}(q)$
is at least
$$q^{n-1}(q-\epsilon1)H(n,q,\epsilon).$$  Since $\GL^\epsilon_{n}(q)$ has
center of order $q-\epsilon1$, the minimal centralizer size of an
element in $\PGL^\epsilon_{n}(q)$ is at least
$q^{n-1}H(n,q,\epsilon)$. There exists a normal subgroup $B$ of
$\Aut(S)$ isomorphic to $\PGL^\epsilon_{n}(q)$ with the property
that $S$ is normal in $B$. Thus the minimal centralizer size in
$\Aut(S)$ of an element in $S$ is at least $q^{n-1}H(n,q,\epsilon)$.
It follows that every $\Aut(S)$-orbit on $S$ has size at most
\[
\frac{|\Aut(S)|}{q^{n-1}H(n,q,\epsilon)}.
\]

On the other hand, by \cite[Lemmas 3.1 and 4.1]{Babai13}, we know
that the proportion of $p$-regular elements in $\PSL^\epsilon_n(q)$
is at least the proportion of elements in $\Sy_n$ that have no
cycles of length divisible by $n$. As the latter proportion is
$(n-1)/n$, we have $|S_{p'}|\geq (n-1)|S|/n$, and it follows from
the conclusion of the previous paragraph that
\[n(\Aut(S),\Cl_{p'}(S))
> \frac{(n-1)|S|}{n |\Aut(S)|} \cdot q^{n-1}H(n,q,\epsilon).\] The
result now follows by $|\Aut(S)| = 2f(n,q-\epsilon1) |S|$.
\end{proof}

\begin{proposition}\label{proposition-linear-unitary}
Let $S=\PSL^\epsilon_n(q)$ for $n\geq 4$ and $\epsilon=\pm$ and let
$p$ be a prime divisor of $|S|$. The number of $\Aut(S)$-orbits on
$p$-regular classes of $S$ is greater than $2\sqrt{p-1}$.
\end{proposition}

\begin{proof} Note that the case $p\mid q$ has been done in
Lemma~\ref{lemma-defining-characteristic}, and so we assume that
$p\nmid q$. Furthermore, if $p\mid |\PSL^\epsilon_{n-1}(q)|$ then we
are done by Lemma~\ref{piPSL} and induction. So we assume also that
$p\mid (q^n-(\epsilon1)^n)$ but $p\nmid (q^i-(\epsilon 1)^i)$ for
every $1\leq i\leq n-1$, which means that $p$ is a primitive prime
divisor of $q^n-1$ when $\epsilon=+$ or $4\mid n$ and $\epsilon=-$,
and a primitive prime divisor of $q^{n/2}-1$ when $n\equiv 2 \pmod
{4}$ and $\epsilon=-$, and a primitive prime divisor of $q^{2n}-1$
when $n$ is odd and $\epsilon=-$.
Lemma~\ref{lemma-linear-unitary-bound} then implies that
\[n(\Aut(S),\Cl_{p'}(S))
> \frac{q^{n-1}(n-1)}{2nf(n,q-\epsilon1)} H(n,q,\epsilon).\]
A straightforward computation shows that this bound is greater than
$2\sqrt{p-1}$, and therefore we are done, unless $q=2$ and $n\leq
9$, or $q=3$ and $n\leq 5$, or $(n,q,\epsilon)\in
\{(4,4,\pm),(4,5,+), (5,4,\pm)\}$.

We now consider these exceptions in a case by case basis.

Let $q=2$. First the cases $S=\SL^\pm_4(2)$, $\SL^\pm_5(2)$, and
$\PSU_6(2)$ can be checked directly using \cite{Atl1}. The case of
$\SL_6(2)$ is not under consideration since $2^6-1$ has no primitive
prime divisor. For $n=7,8,9$ we will show that the number of
different element orders coprime to $p$ is greater than
$2\sqrt{p-1}$, and for that purpose it is enough to, and we will,
assume that $n=7$ as the maximal prime divisor of $S$ is a divisor
of $|\SL^\epsilon_7(2)|$. Here in fact $p=127=2^7-1$ for
$\epsilon=+$ and $p=43=(2^7+1)/3$ for $\epsilon=-$. We consider the
embeddings $\SL^\epsilon_4(2)\times \SL^\epsilon_3(2)\subset
\SL^\epsilon_7(2)$ and $\SL^\epsilon_5(2)\times
\SL^\epsilon_2(2)\subset \SL^\epsilon_7(2)$ and inspect the element
orders in the groups $\SL^\epsilon_k(2)$ for $2\leq k\leq 5$ in
\cite{Atl1} to produce more than $2\sqrt{p-1}$ element orders of
$\SL^\epsilon_7(2)$, proving the desired inequality.

Let $q=3$. Again the case of $S=\PSL^\pm_4(3)$ is available in
\cite{Atl1}, and so we assume that $S=\PSL^\pm_5(3)$. The case
$S=\PSL_5(3)$ is in fact easy as $p=11$ and $\PSL_5(3)=\SL_5(3)$
contains $\SL_4(3)$, which has more than 7 different element orders.
So it remains to consider $S=\PSU_5(3)$, in which case $p$ must be
$61=(3^5+1)/4$. But by using the embedding $\SU_4(3)\subset
\SU_5(3)=\PSU_5(3)$ and inspecting the element orders of $\SU_4(3)$,
we find that $\PSU_5(3)$ has at least 17 element orders coprime to
$p=61$, and hence the bound follows in this case.

Let $(n,q,\epsilon)=(4,4,\pm)$. Then we have $p=17$. From
\cite{Atl1} we observe that $\SL^\epsilon_3(4)$ has more than
$8=2\sqrt{p-1}$ different element orders, and thus, as
$S=\SL^\epsilon_4(4)\geq \SL^\epsilon_3(4)$, it follows that
$n(\Aut(S),\Cl_{p'}(S))>2\sqrt{p-1}$. Similarly when
$(n,q,\epsilon)=(5,4,\pm)$, by using \cite{Atl1} and considering the
embeddings $\SL_3(4)\subset \SL_5(4)=\PSL_5(4)$ and $\SU_3(4)\times
\SL_2(4)\subset \SU_5(4)$ one can produce at least $13$ different
orders coprime to $p$ of $S$, and therefore proving the inequality
as $p\leq 41$ in this case.

Finally for $(n,q,\epsilon)=(4,5,+)$ we have $p=13$ and on the other
hand, using the embeddings $\SL_2(5)\times \SL_2(5)\subset \SL_4(5)$
and $\SL_3(5)\subset \SL_4(5)$ one easily sees that $S$ has element
orders 1, 2, 3, 5, 6, 15, 31, implying that
$n(\Aut(S),\Cl_{p'}(S))\geq 7>2\sqrt{p-1}$.
\end{proof}

\section{Theorem \ref{proposition-simple-groups}: Symplectic and orthogonal
groups}\label{section-symplectic-orthogonal}

The aim of this section is to prove the following theorem, which
implies Theorem~\ref{proposition-simple-groups} for the symplectic
and orthogonal groups.

\begin{theorem}\label{theorem-symplectic-orthogonal-group}
Let $S=\PSp_{2n}(q)$, $\Omega_{2n+1}(q)$ for $n\geq 2$ and
$(n,q)\neq (2,2)$, or $S=P\Omega^\pm_{2n}(q)$ for $n\geq 4$. Let $p$
be a prime divisor of $|S|$. Then
\[
n(\Aut(S),\Cl_{p'}(S))>2\sqrt{p-1},
\]
with a single possible exception of $(S,p)=(\Omega_8^-(4),257)$, in
which case $2\sqrt{p-1}=32\leq n(\Aut(S),\Cl_{p'}(S))$.
\end{theorem}

We note that for the exception $(S,p)=(\Omega_8^-(4),257)$, we found
by using \cite{GAP} that $S$ has exactly $32$ different element
orders coprime to $p$, and therefore it is unlikely that this pair
is a true exception. In any case, since $n(\Aut(S),\Cl_p(S))\geq 1$,
the wanted bound $n(\Aut(S),\Cl_p(S)\cup \Cl_{p'}(S))>2\sqrt{p-1}$
in Theorem~\ref{proposition-simple-groups}(i) still holds for this
exception.

As the cases $p=2, 3$ or $p\mid q$ and $n\geq 3$ have been
considered in Section~\ref{section-some-generalities}, we will
assume that $p\geq 5$. Moreover, we assume $p\nmid q$ except in the
case of $S = \PSp_4(q)\cong \Omega_5(q)$, due to
Lemma~\ref{lemma-defining-characteristic}.

\subsection{Symplectic
groups and odd-dimensional orthogonal groups}

\begin{lemma}
\label{PSp4} Theorem \ref{theorem-symplectic-orthogonal-group} holds
for $S = \PSp_4(q)\cong \Omega_5(q)$ with $q\neq 2$.
\end{lemma}

\begin{proof}
We assume that $q\geq 7$ as the cases $q=3,4,5$ can be confirmed
directly using \cite{Atl1}. First suppose that $p\mid q$. Recall
that $p\geq 5$, and so $q$ is odd. Then $k_{p'}(S)\geq q^2/2$ and
therefore $n(\Aut(S),\Cl_{p'}(S))>q^2/4f$. One can check that
$q^2/(4f)>2\sqrt{p-1}$ for all $q\geq 7$.

So it remains to assume that $p\nmid q$. From
\cite{Enomoto72,Srinivasan68}, we observe that $|\bC_S(g)|\geq
(q^2-1)/(2,q-1)$ for every $g\in S$. Suppose that $p\mid (q^2-1)$.
\cite[Lemma 5.1]{Babai13} then implies that the proportion of
$p$-regular elements in $S$ is at least $3/8$. Therefore
$k_{p'}(S)\geq 3(q^2-1)/8(2,q-1)$, and thus
\[
n(\Aut(S),\Cl_{p'}(S))> \frac{3(q^2-1)}{8f(2,q-1)^2(2,q)}.
\]
This bound is larger than $2\sqrt{q}\geq 2\sqrt{p-1}$ if $q\geq 23$.
When $q<23$ we must have $p\leq 7$ since $p\mid (q^2-1)$ and we
are done as $|S|$ is divisible by at least four primes.

We now consider the case $p\mid (q^4-1)$ but $p\nmid (q^2-1)$. The
proportion of $p$-regular elements in $S$ is now at least $1/2$ again by
\cite[Lemma 5.1]{Babai13}. We then have
\[
n(\Aut(S),\Cl_{p'}(S))> \frac{(q^2-1)}{2f(2,q-1)^2(2,q)}.
\]
This bound is larger than $2q$, which in turn is at least
$2\sqrt{p-1}$ unless \[q\in\{7, 8, 9, 11, 13, 16, 25, 27, 32\}.\]
For each $q$ in this set, we find that the inequality
$$\lceil(q^2-1)/{(2f(2,q-1)^2(2,q))}\rceil > 2\sqrt{p-1}$$ still
holds, unless $(q,p)\in \{(8,13),(9,41)\}$. These two exceptions can
be confirmed using \cite{GAP}.
\end{proof}

Let ${\bf p}(i)$ be the number of distinct ways of representing $i$ as a
sum of positive integers and ${\bf p}_{0}(i)$ be the number of distinct ways
of representing $i$ as a sum of odd positive integers.

\begin{lemma}\label{lemma-kp'-symplectic}
Let $p\geq 3$ be a prime not dividing $q$. We have
\[
k_{p'}(\PSp_{2n}(q))\geq \left\{\begin{array}{ll}\sum_{i=0}^n
{\bf p}(i){\bf p}_{0}(n-i)+
\left\lceil\dfrac{q^n-2}{4n}\right\rceil& \mathrm{ if }\ q \ \mathrm{ is \ odd} \\
{\bf p}(n)+\left\lceil\dfrac{q^n-2}{2n}\right\rceil & \mathrm{if}\ q\
\mathrm{is\ even},\end{array} \right.
\]
and
\[
k_{p'}(\Omega_{2n+1}(q))\geq
\left\{\begin{array}{ll}\sum_{i=0}^{\lfloor (2n+1)/4\rfloor}
{\bf p}(i){\bf p}_{0}(2n+1-4i)+
\left\lceil\dfrac{q^n-2}{4n}\right\rceil& \mathrm{ if }\ q \ \mathrm{ is \ odd} \\
{\bf p}(n)+\left\lceil\dfrac{q^n-2}{2n}\right\rceil & \mathrm{if}\ q\
\mathrm{is\ even}.\end{array} \right.
\]
\end{lemma}

\begin{proof}
Note that, by the assumption, $p$ cannot divide both $q^n-1$ and
$q^n+1$. So let $T$ be a maximal torus of $G:=\Sp_{2n}(q)$ of order
$q^n\pm1$ such that $p\nmid |T|$. Since the fusion of (semisimple)
elements in this torus is controlled by the relative Weyl group of a
Sylow $n$-torus with order $2n$ (see \cite[Proposition
5.5]{Malle-Maroti} and its proof), there exist at least
$(|T|-1)/(2n)$ nontrivial semisimple classes of $G$ with
representatives in $T$. It follows that $S$ has at least
$(q^n-2)/(2n(2,q-1))$ nontrivial $p$-regular semisimple classes.

Suppose first that $q$ is odd. Let $J_k$ denote the Jordan block of
size $k$. By \cite[Proposition 2.3]{Gonshaw-Liebeck17} we know that
for each Jordan form
\[
\sum_{i=1}^r (J_{2k_i})^{a_i}+ \sum_{j=1}^s (J_{2l_j+1})^{b_j},
\]
where the $k_i$ are distinct, the $l_j$ are distinct, and $\sum_{i=1}^r
a_ik_i+\sum_{j=1}^s b_j(2l_j+1)=n$, there are $2^r$ corresponding
unipotent classes of $\Sp_{2n}(q)$ (and therefore of $\PSp_{2n}(q)$)
of elements having that form. These classes are all $p$-regular
since $p\nmid q$. Since there are $\sum_{i=0}^n {\bf p}(i){\bf p}_{0}(n-i)$ such
Jordan forms, we obtain the desired bound in this case.

Symplectic groups in even characteristic and odd-dimensional
orthogonal groups follow from \cite[Proposition 2.4 and Theorem
3.1]{Gonshaw-Liebeck17} in a similar way.
\end{proof}

We are now ready to prove Theorem
\ref{theorem-symplectic-orthogonal-group} for $S=\PSp_{2n}(q)$ and
$S=\Omega_{2n+1}(q)$ for $n\geq 3$. The treatments for these two
families are almost identical, so let us provide details only for
symplectic groups.

Suppose first that $q$ is odd. Then by
Lemma~\ref{lemma-kp'-symplectic} and its proof, we obtain
\[
n(\Aut(S),\Cl_{p'}(S)) \geq 1+\left\lceil \frac{\sum_{i=0}^n
{\bf p}(i){\bf p}_{0}(n-i)-1}{2f}\right\rceil+\left\lceil
\frac{q^n-2}{8fn}\right\rceil=:R(q,n),
\]
as $|\Out(S)|=2f$ and note that the trivial class is an
$\Aut(S)$-orbit itself. Note also that $p\leq (q^n+1)/2$ for $n\geq
4$ and $p\leq q^2+q+1$ for $n=3$. One now can check that
$R(q,n)>2\sqrt{p-1}$ for all relevant values.

Next suppose that $q$ is even. We now have
\[
n(\Aut(S),\Cl_{p'}(S)) \geq 1+\left\lceil
\frac{{\bf p}(n)-1}{f}\right\rceil+\left\lceil
\frac{q^n-2}{2fn}\right\rceil=:R'(q,n),
\]
Again one can check that $\left\lceil
{(q^n-2)}/{(2fn)}\right\rceil>2q^{n/2}\geq 2\sqrt{p-1}$ unless
$(n,q)\in \mathcal{S}:=\{(3,2)$, $(3,4)$, $(4,2)$, $(4,4)$, $(5,2)$,
$(5,4)\}$ or $6\leq n\leq 10$ and $q=2$. We then observe that
$R'(n,q)>2\sqrt{p-1}$ in the latter case.

Note that if $n=5$ then $p\leq (q^5-1)/(q-1)$ and the desired bound
$R'(n,q)>2\sqrt{p-1}$ still holds for $(n,q)=(5,4)$. For
$(n,q)=(5,2)$, the only prime failing the inequality
$R'(n,q)>2\sqrt{p-1}$ is $p=31=q^n-1$, but in this case we remark
that there are at least $\lceil (q^n/10f)\rceil=4$ nontrivial
semisimple $p$-regular classes and at least ${\bf p}(5)-1=6$ nontrivial
unipotent classes, totaling to at least 11 $p$-regular classes of
$S$, and hence $n(\Aut(S),\Cl_{p'}(S))=k_{p'}(S)\geq
11>2\sqrt{p-1}$, as required.

For $n=3$ we note that $p\leq q^2+q+1$ and $\PSp_6(q)$ has at least
9 unipotent classes, and thus the bound holds for $(n,q)=(3,2)$ or
$(3,4)$. For $n=4$ we note that $\PSp_8(q)$ has at least 24
unipotent classes and so we are done for $(n,q)=(4,2)$ and
$(n,q)=(4,4)$ as well. The proof is complete.

%
%
%
%


\subsection{Orthogonal
groups in even
dimension}\label{subsection-orthogognal-even-dimension} Let
$S=P\Omega_{2n}^\epsilon(q)$ for $n\geq 4$, $\epsilon=\pm$, and
$q=\ell^f$ where $\ell$ is a prime.

We start this subsection by proving a lower bound for the number of
unipotent classes in even-dimensional orthogonal groups.

\begin{lemma}\label{lemma-unipotent-class} The following holds:
\begin{itemize}
\item[(i)] $P\Omega_{2n}^\pm(q)$ has at least $\sum_{i=0}^{\lfloor
n/2\rfloor}{\bf p}(i){\bf p}_{0}(2n-4i)$ unipotent classes if $q$ is odd.

\item[(ii)] $P\Omega_{2n}^+(q)$ has at least ${\bf p}(n)+\sum_{i\neq j; i+j\leq n; i,j\ odd}
{\bf p}(n-i-j)$ unipotent classes if $q$ is even.

\item[(iii)] $P\Omega_{2n}^-(q)$ has at least $\sum_{1\leq i\leq n; i\ odd} {\bf p}(n-i)$
unipotent classes if $q$ is even.
\end{itemize}
\end{lemma}

\begin{proof}
First suppose that $q$ is odd. Consider the Jordan forms
\[\sum_{i=1}^r (J_{2k_i+1})^{a_i}+\sum_{j=1}^s (J_{2l_j})^{b_j},\]
where $k_i\geq 0$ are distinct and $l_j\geq 1$ are distinct such
that
\[\sum_{i=1}^r a_i(2k_i+1)+4\sum_{j=1}^s b_j l_j = 2n.\] Here we recall
that $J_k$ denotes the Jordan block of size $k$. It was shown in
\cite[Proposition 2.4]{Gonshaw-Liebeck17} that the unipotent
elements with such a Jordan form  fall into $2^{r-1}$ classes in
each of $\GO_{2n}^+(q)$ and $\GO_{2n}^-(q)$, with the exception that
if $r = 0$, it is 1 class in $\GO_{2n}^+(q)$ and none in
$\GO_{2n}^-(q)$. As $q$ is odd, these classes are inside
$\Omega_{2n}^\pm(q)$ and different classes produce different
corresponding classes of $S$. Now $(i)$ follows since the number of
those Jordan forms is $\sum_{i=0}^{\lfloor n/2\rfloor}{\bf p}(i){\bf p}_{0}(2n-4i)$,
with the remark that there is at least one such form with $r\geq 2$.

In a similar way, (ii) and (iii) follow from the description of
unipotent classes of $\GO_{2n}^\pm(2^f)$ as well as
$P\Omega_{2n}^\pm(2^f)=\Omega_{2n}^\pm(2^f)$ in \cite[Theorem
3.1]{Gonshaw-Liebeck17}.
\end{proof}

\begin{lemma}\label{lemma-orthogonal-bound} Suppose that $p$ is odd and $(n,\epsilon) \neq(4,+)$.
There are at least
\[
1+\left\lceil\frac{q^{n-1}-2}{4f(n-1)(4,q^n-\epsilon1)^2}\right\rceil.
\]
$\Aut(S)$-orbits of $p$-regular semisimple classes of
$S=P\Omega_{2n}^\pm(q)$.
\end{lemma}

\begin{proof}
From the assumption on $p$, we know that $p$ does not divide both
$q^{n-1}-1$ and $q^{n-1}+1$. Let $T$ be a maximal torus of
$G:=\mathrm{Spin}^\epsilon_{2n}(q)$ of order $q^{n-1}\pm1$ such that
$p\nmid |T|$. The relative Weyl group of a Sylow $(n-1)$-torus has
order $2(n-1)$, and thus there are at least $(|T|-1)/(2(n-1))$
nontrivial semisimple classes of $G$ with representatives in $T$. It
follows that $S$ has at least
$(q^{n-1}-2)/(2(n-1)(4,q^n-\epsilon1))$ nontrivial $p$-regular
semisimple classes. The lemma now follows as
$|\Out(S)|=2f(4,q^n-\epsilon1)$.
\end{proof}

Now we are ready to prove Theorem
\ref{theorem-symplectic-orthogonal-group} for even-dimensional
orthogonal groups of rank at least 4. We recall
that $p\geq 5$ and $p\nmid q$.\\


A) Suppose that $p\mid (q^m\pm1)$ for some $m\leq n/2$. One can
check that the bound in Lemma~\ref{lemma-orthogonal-bound} is
greater than $2\sqrt{p-1}$ unless $n=5$ and $q=2, 3, 5, 7, 9$; $n=6$
and $q=2, 3, 5$; or $(n,q)=(7,2), (7,3), (8,2), (8,3)$.

For these exceptional cases, we will prove the desired bound by also
taking into account the unipotent classes of $S$. For instance,
when $(n,q)=(8,3)$ we have that $S$ has at least
$\sum_{i=0}^{4}{\bf p}(i){\bf p}_{0}(16-4i)=69$ unipotent classes by
Lemma~\ref{lemma-unipotent-class}(i), producing at least $\lceil
69/8\rceil=9$ orbits of $\Aut(S)$ on unipotent classes of $S$.
Together with at least 5 orbits on nontrivial $p$-regular semisimple
classes by Lemma~\ref{lemma-orthogonal-bound}, we obtain
$n(\Aut(S),\Cl_{p'}(S))\geq 14>2\sqrt{p-1}$ since $p\leq 41$. All
other cases are treated similarly, except when $(n,q)=(5,2), (5,3)$
or $(5,5)$. If $(n,q)=(5,2)$ or $(5,3)$ then $p$ must be $5$, but
$|S|$ is divisible by at least 4 different primes other than 5, and
thus we still have $n(\Aut(S),\Cl_{p'}(S))\geq 5>2\sqrt{p-1}$.
Finally if $(n,q)=(5,5)$ then $p\leq13$, but $|S|$ is divisible by
at least 7 different
primes, and we are done as well.\\

%
%
%

 B) Suppose that $p$ does not divide $q^i\pm 1$
for every $i\leq n/2$. Let $m$ be minimal subject to the condition
$p\mid q^m\pm 1$. In particular, $m\geq 3$. Using \cite[Lemma
6.1]{Babai13}, we then know that the proportion of $p$-regular
elements in $S$ is at least the proportion of elements in $\Sy_n$
that have no cycles of length divisible by $m$. As the latter
proportion is $(m-1)/m$, we deduce that
\[
|S_{p'}|\geq \frac{m-1}{m}|S|>\frac{n-2}{n}|S|,
\]
where we recall that $S_{p'}$ denotes the set of $p$-regular
elements in $S$.

On the other hand, according to the proof of \cite[Theorem
6.13]{Fulman-Guralnick12}, the centralizer size of an element in
$\SO_{2n}^\pm(q)$ is at least
\[
q^n\left[ \frac{1-1/q}{2^re}\right]^{1/2},
\]
where \[r:=\min\{x\in \NN: \max\{4,\log_q(4n)\}\leq 2^x\}.\] Thus,
for every $g\in S$, we have
\[
|\bC_S(g)|\geq \frac{q^n(2,q^n-\epsilon1)}{2(4,q^n-\epsilon1)}\left[
\frac{1-1/q}{2^re}\right]^{1/2},
\]
and therefore
\[
k_{p'}(S)\geq
\frac{q^n(n-2)(2,q^n-\epsilon1)}{2n(4,q^n-\epsilon1)}\left[
\frac{1-1/q}{2^re}\right]^{1/2}
\]
for every prime $p$, since $|S_{p'}|\geq {(n-2)|S|}/{n}$. It follows
that
\[
n(\Aut(S),\Cl_{p'}(S))\geq
\frac{q^n(n-2)(2,q^n-\epsilon1)}{4fn(4,q^n-\epsilon1)^2}\left[
\frac{1-1/q}{2^re}\right]^{1/2}=:R(n,q)
\]
as $|\Out(S)|=2f(4,q^n-\epsilon1)$.

%

Taking into account the maximum value of $p$ and assuming that $n
\geq 5$, we get $R(n,q) > 2 \sqrt{p-1}$ unless possibly if $n=5$ and
$q \in \{ 2, 3, 4, 5 \}$, 
or $n\in\{6,7\}$ and $q \in \{ 2, 3 \}$, or $(n,q,\epsilon) =
(8,3,+)$, or $q = 2$ and $n\in \{ 8, 9, 10 \}$.

We now use various techniques to prove the inequality
$n(\Aut(S),\Cl_{p'}(S))>2\sqrt{p-1}$ for all these exceptions.

The case $(n,q)=(5,2)$ can be checked using \cite{Atl1}. For
$(n,q)=(5,4)$ we still have $R(n,q) > 2 \sqrt{p-1}$ unless
$p=257=4^4+1$, in which case there are at least $1+\lceil
(4^5-2)/20\rceil=52$ $\Aut(S)$-orbits on $p$-regular semisimple
classes with representatives in a maximal torus of order
$4^5-\epsilon 1$, and thus we are done.

For $(n,q)=(5,3)$ we have $p\leq 61$ and one can check from
\cite{Atl1} that $\mathrm{P\Omega}_{8}^\epsilon(3)$, and therefore
$\mathrm{P\Omega}_{10}^\epsilon(3)$ has more than $2\sqrt{p-1}$ different
element orders coprime to $p$, producing more than $2\sqrt{p-1}$
$\Aut(S)$-orbits on $p$-regular classes of $S$.

Next we suppose $(n,q)=(5,5)$. First we observe that the inequality
$R(n,q)>2\sqrt{p-1}$ still holds unless $p=521=(5^5+1)/6$, which
happens only when $\epsilon=-$, or $p=313=(5^4+1)/2$, which could
happen when either $\epsilon=+$ or $\epsilon=-$. Consider the
natural embeddings
\[
\mathrm{Spin}_6^+(5)\times \mathrm{Spin}_4^+(5)\subset
\mathrm{Spin}_{10}^+(5)
\]
and
\[
\mathrm{Spin}_6^-(5)\times \SL_2(25)\cong \mathrm{Spin}_6^-(5)\times
\mathrm{Spin}_4^-(5)\subset \mathrm{Spin}_{10}^-(5),
\]
which produce the embeddings
\[
\PSL_4(5)\times \PSL_2(5)\times \PSL_2(5)\subset P\Omega_{10}^+(5)
\]
and
\[
\PSU_4(5)\times \PSL_2(25)\subset P\Omega_{10}^-(5),
\]
respectively. Let us first consider $\epsilon=-$. Using the
information on element orders of $\PSU_4(5)$ and $\PSL_2(25)$ in
\cite{Atl1} and \cite{GAP}, we find that the set of all the numbers
of the form $ab$ where $a$ is an element order of $\PSU_4(5)$ and
$b$ is an element order of $\PSL_2(25)$ has cardinality greater than
46. Every element in this set is coprime to both 313 and 521 and is
certainly an element order of $S=P\Omega_{10}^-(5)$ by the above
embedding. It follows that
$n(\Aut(S),\Cl_{p'}(S))>46>2\sqrt{520}\geq 2\sqrt{p-1}$ as wanted.
The case $\epsilon=+$ works in the exact same way.

Now we consider the case $n\in\{6,7\}$ and $q=2$. If $p\neq
127=2^7-1$ then $p\leq 43$. Going through all the cyclic tori of
orders $2^k\pm 1$ which are divisors of $|S|$, we can find more than
$2\sqrt{p-1}$ different element orders, proving the bound. If
$p=127$ then of course $S=\Omega_7^+(2)$, and
Lemma~\ref{lemma-unipotent-class} shows that $S$ has at least
${\bf p}(7)+{\bf p}(3)=18$ nontrivial unipotent classes, producing at least $9$
$\Aut(S)$-orbits on them. Moreover we can find $13$ different orders
of semisimple elements of $S$. Indeed, there are at least three
classes of elements of order $31=2^5-1$ since fusion of classes in a
torus of order 31 are controlled by the relative Weyl group of order
10, and these three classes are contained in at least 2
orbits of $\Aut(S)$. Altogether, we have shown that
$n(\Aut(S),\Cl_{p'}(S))\geq 9+14=23>2\sqrt{126}=2\sqrt{127-1}$, as
desired.

The case of $(n,q)=(6,3)$ is handled similarly to the case
$(n,q)=(5,3)$ by bounding the number of different element orders.
One just examines all the element orders of $\mathrm{P\Omega}_{8}^\epsilon(3)$
using \cite{GAP} together with odd orders of elements in the cyclic
tori (of $\mathrm{Spin}_{12}^\epsilon(3)$) of orders $3^5\pm1$ and
$3^6-\epsilon 1$ to see that the number of different element orders
coprime to $p$ in $S$ is greater than $2\sqrt{p-1}$ for every $p\mid
|S|$.

Next we consider the case $(n,q,\epsilon)=(7,3,\pm)$ or $(8,3,+)$.
The trick in the previous paragraph still works unless
$p=547=(3^7+1)/4$ or $p=1093=(3^7-1)/2$. As in the case of
$(n,q)=(5,5)$, we consider the natural embeddings
\[
\mathrm{Spin}_8^+(3)\times \SL_4(3)\cong \mathrm{Spin}_8^+(3)\times
\mathrm{Spin}_6^+(3)\subset \mathrm{Spin}_{14}^+(3)\subset
\mathrm{Spin}_{16}^+(3)
\]
and
\[
\mathrm{Spin}_8^-(3)\times \SU_4(3)\cong \mathrm{Spin}_8^-(3)\times
\mathrm{Spin}_6^-(3)\subset \mathrm{Spin}_{14}^-(3),
\]
which lead to the embeddings
\[
\mathrm{P\Omega}_8^+(3)\times \PSL_4(3)\subset \mathrm{P\Omega}_{14}^+(3)\subset
\mathrm{P\Omega}_{16}^+(3)
\]
and
\[
\mathrm{P\Omega}_8^-(3)\times \PSU_4(3)\subset \mathrm{P\Omega}_{14}^-(3),
\]
respectively (see \cite[Lemma 2.5]{Nguyen}). We now use the
information on element orders of $P\Omega_8^\pm(3)$ and
$\PSL^\pm_4(3)$ (here $\PSL^-_4(3):=\PSU_4(3)$) in \cite{Atl1,GAP}
to find that the set of element orders of
$P\Omega_{14,16}^\epsilon(3)$ coprime to both 547 and 1093 is
greater than 67. It then follows that
$n(\Aut(S),\Cl_{p'}(S))>67>2\sqrt{1092}\geq 2\sqrt{p-1}$ as wanted.

For $(n,q)=(10,2)$ or $(9,2)$, one just uses Lemmas
\ref{lemma-unipotent-class} and \ref{lemma-orthogonal-bound} to
obtain the desired bound. For $(n,q)=(8,2)$ the same trick works
unless $p=127=2^7-1$ or $p=257=2^8+1$ and $\epsilon=-$. If $p=127$
there are at least $\lceil (2^8)/32\rceil+ \lceil 2^7/28\rceil=13$
$\Aut(S)$-orbits on nontrivial $p$-regular semisimple classes with
representatives in the two tori of coprime orders $2^8-\epsilon1$
and $2^7+1$, and at least $\lceil({\bf p}(7)+{\bf p}(5)+{\bf p}(3))/2\rceil=13$
$\Aut(S)$-orbits on nontrivial unipotent classes, and hence
$$n(\Aut(S),\Cl_{p'}(S))\geq 13+13+1=27>2\sqrt{126}=2\sqrt{p-1}.$$ If
$p=257$ we consider all the nontrivial classes with representatives
in the tori of orders $2^7\pm1, 2^6\pm1$ and $2^5-1$. These orders
are pairwise coprime and so the classes are different and the total
number of $\Aut(S)$-orbits on them is at least
\[
\left\lceil \frac{2^7}{28}\right\rceil+\left\lceil
\frac{2^7-2}{28}\right\rceil+\left\lceil
\frac{2^6}{24}\right\rceil+\left\lceil
\frac{2^6-2}{24}\right\rceil+\left\lceil
\frac{2^5-2}{20}\right\rceil=18.
\]
On the other hand there are at least 13 orbits on nontrivial
unipotent classes, as estimated above, and another orbit on the
classes of elements of order $2^4+1=17$, implying that
$n(\Aut(S),\Cl_{p'}(S))\geq 18+13+1+1=33>2\sqrt{256}=2\sqrt{p-1}$.

The proof of Theorem \ref{theorem-symplectic-orthogonal-group} is
completed by the following two lemmas.

\begin{lemma}\label{lemma-POmega8+}
Theorem \ref{theorem-symplectic-orthogonal-group} holds for
$S=\mathrm{P\Omega}_{8}^+(q)$
\end{lemma}

\begin{proof}
Note that
\[|S|=q^{12}\Phi_1(q)^4\Phi_2(q)^3\Phi_3(q)\Phi_4(q)^2\Phi_6(q)/(4,q^4-1).\]
First we suppose that $p$ divides $\Phi_3(q)$, $\Phi_4(q)$ or
$\Phi_6(q)$. According to \cite[Lemma 6.1]{Babai13}, the proportion
of $p$-regular elements in $S$ is then at least the proportion of
elements in $\Sy_4$ with no cycles of length divisible by $2$, which
is $13/24$. As above, we get
\[
n(\Aut(S),\Cl_{p'}(S))\geq \frac{13q^4}{288f(2,q-1)^3}\left[
\frac{1-1/q}{4e}\right]^{1/2}.
\]
The assumption on $p$ guarantees that this bound is greater than
$2\sqrt{p-1}$, unless $q\in\{2,3,4,5,7,9\}$.

If $q=2$, then $n(\Aut(S), \mathrm{Cl}_{p'}(S)) \geq 24 > 2
\sqrt{p-1}$ by \cite{GAP}, as $p \leq 7$. Let $q = 3$. Since $p \geq
5$, we have $p \in \{ 5, 7, 13 \}$. The number of different orders
of elements in $S$ which are coprime to $p$ is at least $12$ by
\cite{GAP}. Thus $n(\Aut(S), \mathrm{Cl}_{p'}(S)) \geq 12 > 2
\sqrt{13-1}$. Let $q=4$. The set of prime divisors of the order of
$S$ is $\{ 2, 3, 5, 7, 13, 17 \}$. Thus $n(\Aut(S),
\mathrm{Cl}_{p'}(S)) \geq 6$. This forces $p \in \{ 13, 17 \}$.
Using random searches by \cite{GAP} one finds that $S$ contains
elements of orders $15$, $21$, $30$. It follows that $n(\Aut(S),
\mathrm{Cl}_{p'}(S)) \geq 9 > 2\sqrt{p-1}$. Let $q = 5$. The set of
prime divisors of $|S|$ is $\{ 2, 3, 5, 7, 13, 31 \}$. Thus
$n(\Aut(S), \mathrm{Cl}_{p'}(S)) \geq 6$. Assume that $p \in \{ 13,
31 \}$. Using random searches by \cite{GAP}, one finds that $S$
contains elements of orders $10$, $62$, $63$ from which we obtain
the desired bound for $p = 13$, and elements of orders $10$, $26$,
$39$, $63$, $156$, from which we get the bound for $p = 31$. Let $q
= 7$. The set of prime divisors of $|S|$ is $\{ 2, 3, 5, 7, 19, 43
\}$. Thus $n(\Aut(S), \mathrm{Cl}_{p'}(S)) \geq 6$. Assume that $p
\in \{ 19, 43 \}$. Using random searches by \cite{GAP}, one finds
that $S$ contains elements of orders $16$, $25$, $168$, $600$, from
which the desired bound follows for $p = 19$, and $16$, $24$, $25$,
$57$, $168$, $171$, $600$, from which we get the bound for $p = 43$.
Let $q = 9$. The set of prime divisors of $|S|$ is $\{ 2, 3, 5, 7,
13, 41, 73 \}$. Thus $n(\Aut(S), \mathrm{Cl}_{p'}(S)) \geq 7$.
Assume that $p \in \{ 41, 73 \}$. Since $\mathrm{GO}_{6}^{+}(9)
\times \mathrm{GO}_{2}^{+}(9)$ is a subgroup of
$\mathrm{GO}_{8}^{+}(9)$, there is a subgroup of $S$ isomorphic to
$\mathrm{P\Omega}_{6}^{+}(9)\cong \PSL_4(9)$. Using random searches
by \cite{GAP}, one sees that $\mathrm{P\Omega}_{6}^{+}(9)$ (and so
$S$) contains elements of orders $9$, $16$, $20$, $24$, $40$, $60$,
$80$, $91$, $182$, from which the desired bound follows for $p =
41$, and an additional element of order $205$ from which the case $p
= 73$ follows.

Next we suppose that $p$ divides $\Phi_1(q)$ or $\Phi_2(q)$. Arguing
similarly as in Lemma~\ref{lemma-orthogonal-bound}, we then have
that $S$ has at least
\[
1+\left\lceil\frac{q^3-2}{36f(2,q-1)^4}\right\rceil
\]
$\Aut(S)$-orbits on its $p$-regular semisimple classes. This bound
is greater than $2\sqrt{p-1}$ unless $q\in \{2,3,4,5,7,9,11\}$. For
all these exceptions, we have $p\leq 5$ as $p\mid (q^2-1)$, and the
bound easily holds as $|S|$ has at least $5$ different prime
divisors.
\end{proof}

\begin{lemma}\label{lemma-POmega8-}
Theorem \ref{theorem-symplectic-orthogonal-group} holds for
$S=\mathrm{P\Omega}_{8}^-(q)$
\end{lemma}

\begin{proof}
The proof goes along similar lines as in that of
Lemma~\ref{lemma-POmega8+}. Note that
\[|S|=q^{12}\Phi_1(q)^3\Phi_2(q)^3\Phi_3(q)\Phi_4(q)\Phi_6(q)\Phi_8(q)/(4,q^4+1).\]
The case $p$ divides $\Phi_i(q)$ for $i=1,2,3,4$, or $6$ is proved
similarly as we did for $\mathrm{P\Omega}_8^+(q)$. So suppose that $p\mid
\Phi_8(q)$ but $p\nmid \Phi_i(q)$ for all $i=1,2,3,4,6$. According
to \cite[Lemma 6.1]{Babai13}, the proportion of $p$-regular elements
in $S$ is then at least $3/4$. As above, we get
\[
n(\Aut(S),\Cl_{p'}(S))\geq
\frac{3q^4(2,q^4+1)}{16f(4,q^4+1)^2}\left[
\frac{1-1/q}{4e}\right]^{1/2}.
\]
Note that $p\leq q^2+1$ as $p\mid \Phi_8(q)$, and thus this bound is
greater than $2\sqrt{p-1}$ as we wish, unless $q\in\{2,3,4,5,7,9\}$.

If $q = 2$, then $5 \leq p \leq 17$ and $n(\Aut(S),
\mathrm{Cl}_{p'}(S)) \geq 30$ by Gap \cite{GAP}. Let $q = 3$. The
set of prime divisors of the order of $S$ is $\{ 2, 3, 5, 7, 13, 41
\}$. Thus $n(\Aut(S), \mathrm{Cl}_{p'}(S)) \geq 6$. We may assume
that $p \in \{ 13, 41 \}$. Element orders of $S$ include by
\cite{GAP} every positive integer at most $10$ and also $12$, $14$,
$15$. It follows that $n(\Aut(S), \mathrm{Cl}_{p'}(S)) \geq 13 > 2
\sqrt{p-1}$. Let $q = 4$. The set of prime divisors of $|S|$ is $\{
2, 3, 5, 7, 13, 17, 257 \}$. Thus $n(\Aut(S), \mathrm{Cl}_{p'}(S))
\geq 7$. We may assume that $p \in \{ 17, 257 \}$. Random searches
using \cite{GAP} show that element orders of $S$ include $39$ and
$60$. This gives the desired bound for $p = 17$. Assume that $p =
257$. Again by \cite{GAP} we find elements in $S$ of orders at most
$10$ and $12$, $13$, $15$, $17$, $20$, $21$, $30$, $34$, $35$, $39$,
$45$, $51$, $60$, $63$, $65$, $85$, $102$, $105$, $170$, $195$,
$255$, $315$. Thus $n(\Aut(S), \mathrm{Cl}_{p'}(S)) \geq 2
\sqrt{p-1}$.  Note that we do not have strict inequality here and
this possible exception is listed in Table~\ref{table-exceptions}.

Let $q = 5$. The set of prime divisors of $|S|$ is $\{ 2, 3, 5, 7,
13, 31, 313 \}$. Thus $n(\Aut(S), \mathrm{Cl}_{p'}(S)) \geq 7$. We
may assume that $p \in \{ 31, 313 \}$. A random search using
\cite{GAP} shows that there is an element of order $104$ in $S$ and
thus also elements of orders $4$, $8$, $26$, $52$. We may thus
assume that $p = 313$. Several random searches give element orders
$24$, $12$, $36$, $60$, $15$, $30$, $6$, $20$, $252$, $84$, $42$,
$21$, $126$, $63$, $9$, $130$, $65$, $260$, $124$, $62$, $93$,
$372$, $186$, $312$. From this we obtain $n(\Aut(S),
\mathrm{Cl}_{p'}(S)) > 2 \sqrt{313-1}$. Let $q = 7$. The set of
prime divisors of $|S|$ is $\{ 2, 3, 5, 7, 19, 43, 1201 \}$. This
gives $n(\Aut(S), \mathrm{Cl}_{p'}(S)) \geq 7$ and so $p \in \{ 19,
43, 1201 \}$. Further element orders of $S$ (larger than $1$ and
different from a prime) include $4$ and $8$, from which the desired
bound follows for $p = 19$. Even further element orders are $10$,
$20$, $40$, $50$, from which the case $p = 43$ is completed. Many
\cite{GAP} searches give a total of $70$ different element orders in
$S$ each of which is coprime to $1201$. This finishes the proof of
the bound for $p = 1201$.

Finally let $q = 9$. The set of prime divisors of $|S|$ is $\{ 2, 3,
5, 7, 13, 17, 41, 73, 193 \}$. It follows that $n(\Aut(S),
\mathrm{Cl}_{p'}(S)) \geq 9$ and so $p \in \{ 41, 73, 193 \}$. As in
the proof of Lemma \ref{theorem-symplectic-orthogonal-group}, the
group $\mathrm{P\Omega}_{6}^{+}(9)$ may be viewed as a subgroup of
$S$. Element orders in $\mathrm{P\Omega}_{6}^{+}(9)$ include $4$,
$8$, $16$, $91$, $182$. The desired bound follows for $p = 41$.
Further element orders in $\mathrm{P\Omega}_{6}^{+}(9)$ are $205$,
$15$, $40$. The bound follows for $p = 73$. Further element orders
in $\mathrm{P\Omega}_{6}^{+}(9)$ include $6$, $10$, $12$, $20$,
$24$, $26$, $30$, $60$, $80$. The group
$\mathrm{P\Omega}_{6}^{-}(9)$ is also a subgroup of $S$ and two of
its element orders are $328$ and $365$. From these we get the
desired bound for $p = 193$.
\end{proof}


\section{Theorem \ref{proposition-simple-groups}: Exceptional
groups}\label{section-exceptional}

In this section $S$ will be a simple exceptional group of Lie type,
but not the Tits group. Then $S$ is of the form $G/\bZ(G)$, where
$G=\GC^F$ is the set of fixed points of a simple algebraic group of
simply connected type, under a Frobenius endomorphism $F$ associated
to the field of $q=\ell^f$ elements, where $\ell$ is a prime. Let
$r$ be the rank of $\GC$, which we also call the rank of $S$. We
then have $\bC_{G}(g)\geq (q-1)^r$ for all $g\in G$, which implies
that $\bC_{S}(g)\geq (q-1)^r/|\bZ(G)|$ for all $g\in S$.

Let $\Phi_n(q)$ denote the value of the $n$-cyclotomic polynomial at
$q$ and let $p(S)$ be the maximal prime divisor of $S$. Also, let
$\Phi_4^{\pm}(q):=q\pm \sqrt{2q}+1$ for $q=2^{2m+1}$,
$\Phi_6^{\pm}(q):=q\pm \sqrt{3q}+1$ for $q=3^{2m+1}$, and
$\Phi_{12}^{\pm}(q):=q^2\pm \sqrt{2q^3}+q\pm \sqrt{2q}+1$ for
$q=2^{2m+1}$. Using the order formulas of $S$ \cite{Atl1}, one can
find an upper bound for $p(S)$, which we record in
Table~\ref{table-upper-bound-p(s)}. We observe that in all the cases
$p(S)\leq \Psi_S(q)$ for a polynomial $\Psi_S$ of degree at most
$r(S)$. In fact, $\deg(\Psi_{n(S)})= r$ in all cases except
$S=E_7(q)$.

\begin{table}[ht]
\caption{Upper bounds for $p(S)$.\label{table-upper-bound-p(s)}}
\begin{center}
\begin{tabular}{cc}
\hline
$S$  & upper bound for $p(S)$ \\
\hline
$\ta B_2(q), q=2^{2m+1}$ & $\Phi^+_4(q)$\\
 $G_2(q)$ & $\Phi_3(q)$\\
 $\ta G_2(q), q=3^{2m+1}$& $\Phi^+_6(q)$\\
 $F_4(q)$&$\Phi_8(q)$\\
 $\ta F_4(q), q=2^{2m+1}$& $\Phi^+_{12}(q)$\\
 $\tb D_4(q)$& $\Phi_{12}(q)$\\
 $E_6(q)$& $\Phi_9(q)$\\
 $\ta E_6(q)$ & $\Phi_{18}(q)$\\
 $E_7(q)$ & $\Phi_{7}(q)$\\
 $E_8(q)$& $\Phi_{30}(q)$\\
 \hline
\end{tabular}
\end{center}
\end{table}

We recall from Section~\ref{section-some-generalities} that a
subgroup $T$ of $G$ is said to be strongly self-centralizing if
$\bC_G(t)=T$ for every $1\neq t\in T$. It turns out that every group
of exceptional types other than $E_7(q)$ has one or more strongly
self-centralizing torus, as worked out in \cite{Babai-Palfy-Saxl09}.
This information is collected in
Table~\ref{table-self-centralizing-tori} for convenient reference.\\

We now examine each family of groups separately.\\

\begin{table}[ht]
\caption{Strongly self-centralizing maximal tori of exceptional
groups.\label{table-self-centralizing-tori}}
\begin{center}
\begin{tabular}{cccc}
\hline
$S$  & conditions& $|T|$& $|\bN_S(T)/T|$ \\
\hline
$\ta B_2(q)$ &$q=2^{2m+1}$ & $\Phi^\pm_4(q)$& 4\\
 $G_2(q)$& \begin{tabular}{l} $q\not\equiv 1(\bmod 3)$\\ $q\not\equiv 2(\bmod 3)$
 \end{tabular} & \begin{tabular}{l} $\Phi_3(q)$\\ $\Phi_6(q)$
 \end{tabular} & 6\\

$\ta G_2(q)$& $q=3^{2m+1}$ & $\Phi_6^\pm(q)$ &$6$\\

$F_4(q)$ & & $\Phi_{12}(q)$ & 12\\

$\ta F_4(q)$ & $q=2^{2m+1}$ & $\Phi_{12}^\pm(q)$ &$12$\\

$\tb D_4(q)$ & & $\Phi_{12}(q)$ & 4\\

$E_6(q)$ & & $\Phi_{9}(q)/(3,q-1)$ & 9\\

$\ta E_6(q)$ & & $\Phi_{18}(q)/(3,q+1)$ & 9\\

$E_8(q)$ & & \begin{tabular}{l} $\Phi_{24}(q)$\\ $\Phi_{15}(q)$\\
$\Phi_{30}(q)$ \end{tabular} & \begin{tabular}{l} 24\\ 30\\
30
 \end{tabular}\\
\hline
\end{tabular}
\end{center}
\end{table}

1) $S=\ta B_2(q)$ with $q=2^{2m+1}\geq 2^3$. We will assume that
$m\geq 3$ as information on $\ta B_2(8)$ and $\ta B_2(32)$ are
available in \cite{Atl1}. As seen in
Table~\ref{table-self-centralizing-tori}, $S$ contains two strongly
self-centralizing tori $T_1$ and $T_2$ of order $\Phi_4^\pm(q)$ such
that $|\bN_S(T_i)/T_i|=4$. Assume first that $p$ divides either
$\Phi_4^+(q)$ or $\Phi_4^-(q)$. We then have $|S_{p'}|> {3|S|}/{4}$
by Lemma~\ref{lemma-strongly-self-centralizing}(i). Therefore,
\[
k_{p'}(S)> \frac{3(q-1)}{4},
\]
which implies that
\[
n(\Aut(S), \Cl_{p'}(S))>\frac{3(q-1)}{4(2m+1)},
\]
as $|\Out(S)|=2m+1$. We observe that
$$3(q-1)/(4(2m+1))>2\sqrt{\Phi_4^+(q)-1}\geq 2\sqrt{p-1}$$ for all
$q\geq 2^9$. So we are done unless $q=2^7$. In fact, when $q=2^7$ we
still have $3(q-1)/(4(2m+1))>2\sqrt{p-1}$ unless $p=113$.

Next we assume $p\nmid \Phi_4(q)=\Phi_4^+(q)\cdot\Phi_4^-(q)$, which
means that $p\mid 2(q-1)$. By
Lemma~\ref{lemma-strongly-self-centralizing}(ii), we have $
|S_{p'}|> {2|S|}/{5}$. Therefore, $k_{p'}(S)> {2(q-1)}/{5}$, and it
follows that $n(\Aut(S), \Cl_{p'}(S))>{2(q-1)}/{5(2m+1)}$. As $p\leq
q-1$, we deduce that $n(\Aut(S), \Cl_{p'}(S))>2\sqrt{p-1}$ for
$q\geq 2^{13}$. Indeed, we have $p\leq 89$ when $q=2^{11}$ and
$p\leq 73$ when $q=2^9$ and thus the desired bound still holds in
those cases. We are left with $(S,p)=(\ta B_2(2^7),127)$.

Let $S = \ta B_2(2^7)$. We have $|\Out(S)| = 7$. Suzuki
\cite{Suzuki} proved that $S$ has $2^{7}+3$ conjugacy classes. The
trivial element of $S$ forms a single $\Aut(S)$-orbit. The group $S$
has $1$ class of involutions. It has $2$ classes of elements of
order $4$ accounting for $2$ $\Aut(S)$-orbits. The group $S$ has
$63$ conjugacy classes of elements of order $127$ accounting for at
least $9$ $\Aut(S)$-orbits. It has $28$ conjugacy classes of
elements of order $113$ forming at least $4$ orbits. There are $36$
classes of elements in a cyclic torus of order $145$, where $1$ of
them is for elements of order $5$, $7$ of them for elements of order
$29$, and $28$ of them for elements of order $145$. These give at
least $6$ orbits. Adding all these together we have $n(\Aut(S),
\mathrm{Cl}_{p'}(S)) \geq 14$ for the pair $(S,p)=(\ta
B_2(2^7),127)$ and $n(\Aut(S), \mathrm{Cl}_{p'}(S)) \geq 19$ for the
pair $(S,p)=(\ta B_2(2^7),113)$, as appearing in
Table~\ref{table-exceptions}. We also note that for these two
exceptions, $n(\Aut(S), \mathrm{Cl}_{p'}(S))\geq 14>2(p-1)^{1/4}$
and $n(\Aut(S),\Cl_p(S)\cup
\mathrm{Cl}_{p'}(S))=23>2\sqrt{p-1}$.\\

2) $S=G_2(q)$ with $q=\ell^f\geq 3$. We will assume that $q\geq 7$
as the cases $q=3, 4, 5$ are available in \cite{Atl1}. First we
consider $p\mid \Phi_3(q)$. If $q\not\equiv 1 \pmod 3$ then $S$ has
a strongly self-centralizing torus $T$ of order $\Phi_3(q)$ such
that $|\bN_S(T)/T|=6$. Lemma~\ref{lemma-strongly-self-centralizing}
then implies that $|S_{p'}|>5|S|/6$. Thus $n(\Aut(S),
\Cl_{p'}(S))>5(q-1)^2/{(6fg)}$ for $g=2$ if $\ell=3$ and $g=1$
otherwise. Therefore $n(\Aut(S), \Cl_{p'}(S))>
2\sqrt{\Phi_3(q)-1}\geq 2\sqrt{p-1}$ unless $q=9$, but in this
exceptional case $p$ is at most $13$ and the bound still follows.

So assume that $q\equiv 1 \pmod 3$. Then $S$ has a maximal torus of
order $\Phi_6(q)$ with the relative Weyl group of order 6 as above,
implying that there are $(q^2-q)/6$ classes of $S$ with
representatives being nontrivial elements in this torus. Consider
another torus of order $(q+1)^2$ with the relative Weyl group of
order 12, we find another $(q^2+2q)/12f$ nontrivial different
classes. We now have \[k_{p'}(S)\geq
1+\frac{q^2-q}{6}+\frac{q^2+2q}{12},\] and thus
\[
n(\Aut(S), \Cl_{p'}(S))\geq 1+\frac{q^2-q}{6f}+\frac{q^2+2q}{12f}.
\]
This bound is greater than $2\sqrt{\Phi_3(q)-1}\geq 2\sqrt{p-1}$,
and hence we are done, unless $q=7$. When $q=7$ we have $p= 19$ and
the above bound is still greater than $2\sqrt{p-1}$.

The case $p\mid \Phi_6(q)$ is similar, so suppose that $p\nmid
\Phi_3(q)\Phi_6(q)$. Then by
Lemma~\ref{lemma-strongly-self-centralizing} we have $n(\Aut(S),
\Cl_{p'}(S))>(q-1)^2/{(7fg)}$ for $g=2$ if $\ell=3$ and $g=1$
otherwise, but now $p\leq q+1$. One can check that
$(q-1)^2/{(7fg)}>2\sqrt{q}\geq 2\sqrt{p-1}$ unless $q=7$ or $9$. But
in those cases we have $p\leq 7$ or $5$, respectively, and hence we
still have $n(\Aut(S), \Cl_{p'}(S))>
2\sqrt{p-1}$.\\

3) $S=\ta G_2(q)$ with $q=3^{2m+1}\geq 3^3$. We assume $m\geq 2$ as
the case $\ta G_2(27)$ can be confirmed directly using \cite{Atl1}.
We know that $S$ contains two strongly self-centralizing tori $T_1$
and $T_2$ of order $\Phi_6^\pm(q)$ such that $|\bN_S(T_i)/T_i|=6$.
Assuming that $p$ divides either $\Phi_6^+(q)$ or $\Phi_6^-(q)$, we
have $k_{p'}(S)> {5(q-1)}/{6},$ which implies that
\[
n(\Aut(S), \Cl_{p'}(S))>\frac{5(q-1)}{6(2m+1)}.
\]
Now we observe that $5(q-1)/(6(2m+1))>2\sqrt{\Phi_6^+(q)-1}\geq
2\sqrt{p-1}$ for all $q\geq 3^5$.

Next we assume $p\nmid \Phi_6(q)=\Phi_6^+(q)\cdot\Phi_6^-(q)$, which
yields $p\mid 3(q^2-1)$. By Lemma
\ref{lemma-strongly-self-centralizing}(ii), we have $ |S_{p'}|>
{2|S|}/{7}$. Therefore, $k_{p'}(S)> {2(q-1)}/{7}$, implying that
$n(\Aut(S), \Cl_{p'}(S))>{2(q-1)}/{7(2m+1)}$. As $p\leq (q-1)/2$, we
deduce that $n(\Aut(S), \Cl_{p'}(S))>2\sqrt{p-1}$ for $q\geq 3^7$.
When $q=3^5$ we in fact have $p\leq
61$ and thus the desired bound also holds.\\

4) $S=F_4(q)$ with $q=\ell^f$. Note that $S$ has a strongly
self-centralizing torus $T$ of order $\Phi_{12}(q)$ and
$|\bN_S(T)/T|=12$. Therefore, if $p\mid \Phi_{12}(q)$, we have
\[n(\Aut(S), \Cl_{p'}(S))>\frac{11(q-1)^4}{12f(2,q)},\] which is larger than
$2\sqrt{\Phi_{12}(q)-1}\geq 2\sqrt{p-1}$ unless $q=2$ or $3$. When
$q=2$, there are at least $18$ $\Aut(S)$-orbits on unipotent classes
of $S$ while $p\leq 17$ by \cite{Atl1}, so we are done. When $q=3$
then $p=\Phi_{12}(3)=73$, and
$n(\Aut(S), \Cl_{p'}(S))=k_{p'}(S)$. 
 Note that $G=S$ has a maximal torus of order $\Phi_8(q)=82$ with the
relative Weyl group of order 8, so we have at least 11 classes of
elements in this torus. Similarly, there are at least
$(\Phi_3(q)^2-1)/24=168/24=7$ classes of nontrivial elements in a
maximal torus of order $\Phi_3(q)^2$. We now have $n(\Aut(S),
\Cl_{p}(S)\cup \Cl_{p'}(S))\geq 18>2\sqrt{73-1}$, as wanted.

So we assume $p\nmid \Phi_{12}(q)$. Then $S$ has at least
$(q^4-q^2)/12$ nontrivial classes with representatives in a maximal
torus of order $\Phi_{12}(q)$. It follows that \[n(\Aut(S),
\Cl_{p'}(S))>1+\left\lceil\frac{q^4-q^2}{12f(2,q)}\right\rceil.
\] Note that $p\leq
\Phi_8(q)=q^4+1$ and therefore we are done unless $q= 2, 3, 4, 8$.
The cases $q=2$ or $3$ can be handled as above. When $q=8$ we have
$p\leq 241$ and thus $1+{(q^4-q^2)}{(12f(2,q))}$ is still greater
than $2\sqrt{p-1}$. When $q=4$,
we also have the desired bound unless $p=4^4+1=257$.

Suppose $(S,p)=(F_4(4),257)$. From the Dynkin diagrams, we know that
the groups $\PSL_3(4) \times \PSL_3(4)$, $\SU_3(4) \times \SU_3(4)$,
and $\Sp_6(4)$ are all sections of $S$, and thus every element
orders of these groups are element orders of $S$. Moreover every
element order of $F_4(2)$ is also an element order of $S$. Using
\cite{Atl1,GAP} to examine the element orders of these smaller
groups, we find that they altogether have way more than
$32=2\sqrt{p-1}$ different element
orders coprime to $p=257$, proving the theorem in this case.\\

5) $S=\ta F_4(q)$ with $q=2^{2m+1}\geq 8$. We know that $S$ contains
two strongly self-centralizing tori $T_1$ and $T_2$ of order
$\Phi_{12}^\pm(q)$ such that $|\bN_S(T_i)/T_i|=12$. Assume that $p$
divides either $\Phi_{12}^+(q)$ or $\Phi_{12}^-(q)$. Then
$k_{p'}(S)> {11(q-1)^2}/{12}$, implying that
\[
n(\Aut(S), \Cl_{p'}(S))>\frac{11(q-1)^2}{12(2m+1)}.
\]
One can check that
$11(q-1)^2/(12(2m+1))>2\sqrt{\Phi_{12}^+(q)-1}\geq 2\sqrt{p-1}$ for
all $q\geq 8$ and we are done.

Next we assume $p\nmid
\Phi_{12}(q)=\Phi_{12}^+(q)\cdot\Phi_{12}^-(q)$. Then $p\mid
q\Phi_1(q)\Phi_2(q)\Phi_4(q)\Phi_6(q)$ and hence $p\leq q^2+1$. We
now have $|S_{p'}|> {2|S|}/{13}$. Therefore, $k_{p'}(S)>
{2(q-1)^2}/{13}$, which implies that $n(\Aut(S),
\Cl_{p'}(S))>{2(q-1)^2}/{13(2m+1)}$. We deduce that $n(\Aut(S),
\Cl_{p'}(S))>2\sqrt{p-1}$ for $q\geq 2^7$. Indeed, we have $p\leq
61$ and $p\leq 19$ when $q=2^5$ and $2^3$,
respectively, and thus we are also done in these cases.\\

6) $S=\tb D_4(q)$ with $q=\ell^f$. We assume $q\geq 3$ as the case
$\tb D_4(2)$ is available in \cite{Atl1}. We know that $S$ has a
strongly self-centralizing tori $T$ of order $\Phi_{12}(q)$ such
that $|\bN_S(T)/T|=4$. Assume first that $p\mid \Phi_{12}(q)$. Then
$n(\Aut(S), \Cl_{p'}(S))>{3(q-1)^4}/{(12f)}$. It then follows that
$n(\Aut(S), \Cl_{p'}(S))>2\sqrt{p-1}$ unless $q=3$ or $4$. For $q=3$
we have $p=\Phi_{12}(3)=73$ and note that as $C_7\times \SU_3(3)$
and $C_{13}\times \SL_3(3)$ are sections of $S$, we can find more
than $17>2\sqrt{72}$ different element orders coprime to $73$ in
$S$. For $q=4$ we have $p=\Phi_{12}(4)=241$, and we can find by
\cite{GAP} and using the paper \cite{Kleidman} at least
$31>2\sqrt{p-1}$ different element orders from the subgroups $\tb
D_4(2)$, $G_2(4)$, $C_{13}\times \SU_3(4)$, $\PGL_3(4)$, $\PSL_2(64)
\times \PSL_2(4)$ and $C_{21} \circ \SL_{3}(4)$ of $S$.

Next assume that $p\nmid \Phi_{12}(q)$. Then $p\mid
q\Phi_1(q)\Phi_2(q)\Phi_3(q)\Phi_6(q)$ and thus $p\leq q^2+q+1$. Now
there are at least $(q^4-q^2)/4$ nontrivial classes with
representatives in a maximal torus of order $\Phi_{12}(q)$, and
therefore $n(\Aut(S), \Cl_{p'}(S))>1+{(q^4-q^2}/{(12f)}$, which is
larger
than $2\sqrt{q^2+q}\geq 2\sqrt{p-1}$ for all $q\geq 3$.\\

7) $S=E_6(q)$ with $q=\ell^f$. We again assume $q\geq 3$ as the case
$E_6(2)$ is available in \cite{Atl1}. We know that $S$ has a strongly
self-centralizing tori $T$ of order $\Phi_{9}(q)/(3,q-1)$ such that
$|\bN_S(T)/T|=9$. Assume that $p\mid |T|$. Recall that $\bC_S(g)\geq
(q-1)^6/(3,q-1)$ for all $g\in S$ and $|\Out(S)|=2f(3,q-1)$. Now we
have
\[n(\Aut(S), \Cl_{p'}(S))>\frac{4(q-1)^6}{9f(3,q-1)^2}.\]
It turns out that $4(q-1)^6/9f(3,q-1)^2>2\sqrt{p-1}$ unless $q=3$
or $4$. When $q=4$ we have $p\leq 73$ so the bound $n(\Aut(S),
\Cl_{p'}(S))>2\sqrt{p-1}$ still holds. When $q=3$ the only prime we
need to check is $p=757=\Phi_9(3)$. So let $(S,p)=(E_6(3),757)$. The
union of the set of prime divisors of $|E_6(3)|$ with the set of
element orders of the sections $\PSL_{6}(3)$, $\PSL_{3}(3) \times
\PSL_{3}(3) \times \PSL_{3}(3)$, $\PSL_{3}(27)$ and $\PSL_{5}(3)
\times \PSL_{2}(3)$ consists of $47$ integers none of which is
divisible by $757$. The group $\mathrm{P\Omega}_{10}^{+}(3)$ is also
a section of $E_{6}(3)$ and by finding orders of random elements in
$\mathrm{P\Omega}_{10}^{+}(3)$ using \cite{GAP}, we may obtain $8$
extra integers, namely $21$, $35$, $45$, $70$, $82$, $84$, $90$ and
$164$, apart from the $47$ previously found. Finally, $55 > 2
\sqrt{757-1}$.

Suppose $p\nmid |T|$, and so we have $p\leq q^4+1$. Let $G$ be the
extension of $S$ be diagonal automorphisms. Then $G$ has a maximal
torus of order $\Phi_{9}(q)=q^6+q^3+1$ with the relative Weyl group
of order 9. Therefore there are at least $(q^6+q^3)/9$ nontrivial
classes with representatives in that torus. Therefore $G$ has at
least $(q^6+q^3)/9$ nontrivial $p$-regular classes, and thus, by the
orbit counting formula,
\[n(\Aut(S), \Cl_{p'}(S))\geq1+\frac{q^6+q^3}{18f(3,q-1)},\] which is
larger
than $2q^2\geq 2\sqrt{p-1}$ for all $q\geq 3$.\\

8) $S=\ta E_6(q)$ with $q=\ell^f$. We assume $q\geq 3$ as the case
$E_6(2)$ is available in \cite{Atl1}. From
Table~\ref{table-self-centralizing-tori}, we know that $S$ has a
strongly self-centralizing torus $T$ of order $\Phi_{18}(q)/(3,q+1)$
such that $|\bN_S(T)/T|=9$. Suppose that $p\mid |T|$. Then
\[n(\Aut(S),
\Cl_{p'}(S))\geq\frac{k_{p'}(S)}{2f(3,q+1)}>\frac{4(q-1)^6}{9f(3,q+1)^2}\geq
2\sqrt{\Phi_{18}(q)/(3,q+1)-1}\] unless $q=3$. When $q=3$ we have
$p=19$ or $37$ and the bound $n(\Aut(S), \Cl_{p'}(S))>2\sqrt{p-1}$
still holds.

Next we assume $p\nmid |T|$. Then $p\leq q^4+1$. As in the case of
$E_6(q)$, we get
\[n(\Aut(S), \Cl_{p'}(S))\geq1+\frac{q^6-q^3}{18f(3,q+1)},\] which is
larger than $2q^2\geq 2\sqrt{p-1}$ for all $q\geq 3$.\\

9) $S=E_7(q)$ with $q=\ell^f$. As the case $q=2$ is available in
\cite{Atl1}, we assume that $q\geq 3$. This is the only family of
exceptional groups that do not always possess a strongly
self-centralizing maximal torus. It was shown in \cite[Theorem
3.4]{Babai-Palfy-Saxl09} that the proportion of $p$-regular elements
in $S$ is at least $1/15$, for every prime $p$. Therefore
\[n(\Aut(S), \Cl_{p'}(S))>\frac{(q-1)^7}{15f(2,q-1)^2}=:R(q).\]

Recall from Table~\ref{table-upper-bound-p(s)} that $p\leq
\Phi_7(q)=(q^7-1)/(q-1)$ and one can check that $R(q)>
2\sqrt{\Phi_7(q)-1}$ for all $q\geq 5$. When $q=4$ we observe that
the largest prime divisor of $S$ is 257 and hence the inequality
$R(q)> 2\sqrt{p-1}$ still holds. For $q=3$ the bound is also good
unless $p=757=\Phi_9(3)$ or $1093=\Phi_7(3)$. The case of
$(S,p)=(E_7(3),757)$ follows from the case $(S,p)=(E_6(3),757)$ we
already examined above. Finally let $(S,p)=(E_7(3), 1093)$. We know
that $E_6(3)$, and hence $E_7(3)$, has at least 55 different element
orders coprime to 757 as well as 1093. On the other hand elements in
a maximal torus of $(E_7)_{ad}(3)$ of order
$\Phi_9(3)\Phi_1(3)=2\cdot 757$ are controlled by its relative Weyl
group of order 18, implying that there are at least 42 classes of
elements of order 757 with representatives in this torus, which in
turn produce at least 21 $\Aut(S)$-orbits on those classes. We now
have at least $55+21 > 2\sqrt{1092}$, orbits of $\Aut(S)$ on
$p$-regular classes of $S$.\\


10) $S=E_8(q)$ with $q=\ell^f\geq 3$ as the case $q=2$ can be
checked directly using \cite{Atl1}. First assume that $p\mid
\Phi_{15}(q)$ or $p\mid \Phi_{30}(q)$. Since $S$ contains two
strongly self-centralizing maximal tori $T_1$ and $T_2$ of order
$\Phi_{15}(q)$ and $\Phi_{30}(q)$ such that $|\bN_S(T_i)/T_i|=30$,
we obtain $k_{p'}(S)> 29(q-1)^8/30$ and thus
\[n(\Aut(S),\Cl_{p'}(S))\geq \frac{29(q-1)^8}{30f},\] which turns out to be
greater than $2\sqrt{\Phi_{30}(q)-1}\geq 2\sqrt{p-1}$ for all $q\geq
3$.

Next we suppose $p\mid \Phi_{24}(q)$. As $S$ contains one strongly
self-centralizing maximal torus $T$ of order $\Phi_{24}(q)$ such
that $|\bN_S(T)/T|=24$, we deduce that $k_{p'}(S)> 23(q-1)^8/24$ and
thus $n(\Aut(S),\Cl_{p'}(S))\geq 23(q-1)^8/(24f)$, which again is
greater than $2\sqrt{q^8-q^4}\geq 2\sqrt{p-1}$ for all $q\geq 3$.

Lastly we assume $p\nmid \Phi_{15}(q)\Phi_{24}(q)\Phi_{30}(q)$. Then
$p\leq \Phi_7(q)=(q^7-1)/(q-1)$. By
Lemma~\ref{lemma-strongly-self-centralizing} we get
\[
|S_{p'}|>\left(\frac{1}{25}+\frac{1}{31}+\frac{1}{31}\right)|S|=\frac{81}{775}|S|.
\]
Thus $k_{p'}(S)> 81(q-1)^8/775$ and hence
$n(\Aut(S),\Cl_{p'}(S))\geq 81(q-1)^8/(775f)$. One can check that
$81(q-1)^8/(775f)> 2\sqrt{\Phi_7(q)-1}$ for all $q\geq 4$. For the
last case $q=3$, we note that $S$ has at least $\Phi_{30}(3)/30$
conjugacy classes with representatives in the maximal torus of order
$\Phi_{30}(3)$, which implies that
$$n(\Aut(S),\Cl_{p'}(S))=k_{p'}(S)>2\sqrt{1092}\geq 2\sqrt{p-1}$$ for
every prime divisor of $|S|$ not dividing
$\Phi_{15}(3)\Phi_{24}(3)\Phi_{30}(3)$.

We have finished the proof of
Theorem~\ref{proposition-simple-groups}.


\section{Bounding the number of $p$-regular classes in finite groups of Lie type}
\label{section-Lie-type-bound}

The following is Theorem~\ref{theorem-p-regular-bound-Lietype} in
the introduction.

\begin{theorem}\label{theorem-p-regularagain}
Let $S$ be a simple group of Lie type defined over the field of $q$
elements with $r$ the rank of the ambient algebraic group. We have
\[
k_{p'}(S)> \frac{q^r}{17r^2}
\]
for every prime $p$.
\end{theorem}

\begin{proof} The theorem is known in the case $p\nmid |S|$, as we
already mentioned in the introduction that $k(S)>q^r/d$ where $d$ is
the order of the group of diagonal automorphisms of $S$ and the
values of $d$ for various groups are known, see \cite{Atl1} for
instance. We are also done in the case $p$ is the defining
characteristic of $S$, by using the same arguments as in the proof
of Lemma~\ref{lemma-defining-characteristic}. For $S$ an exceptional
group, the theorem follows from our work in
Section~\ref{section-exceptional}. Here we note that all the bounds
obtained are of the form $c(q-1)^r$ where $c$ is a constant
depending on the rank $r$ only. It is then straightforward to check
that $c(q-1)^r>q^r/(17r^2)$ for all types and all $q>2$. The case
$q=2$ can be proved by a direct check using \cite{Atl1,GAP}.

Suppose $S=\PSp_{2r}(q)$ or $\Omega_{2r+1}(q)$. The case of odd $p$
follows from Lemma~\ref{lemma-kp'-symplectic}. When $p=2$ similar
arguments as in the proof of Lemma~\ref{lemma-kp'-symplectic} apply,
with the remark that either $(q^r-1)/2$ or $(q^r+1)/2$ is odd (when
$q$ is odd), and hence the bound is $k_{p'}(S)>q^r/8r>q^r/(17r^2)$.

Let $S=\mathrm{P\Omega}_{2r}^\pm(q)$. From
Subsection~\ref{subsection-orthogognal-even-dimension} we know that
the minimum centralizer size of an element in $S$ is at least
\[
\frac{q^r(2,q^r-\epsilon1)}{2(4,q^r-\epsilon1)}\left[
\frac{1-1/q}{2^ke}\right]^{1/2},
\]
where $k:=\min\{x\in \NN: \max\{4,\log_q(4r)\}\leq 2^x\}$. On the
other hand, By \cite[Theorem 1.1]{Babai-Palfy-Saxl09}, the
proportion of $p$-regular elements in $S$ is at least $1/4r$. We
deduce that
\[
k_{p'}(S)\geq \frac{q^r(2,q^r-\epsilon1)}{8r(4,q^r-\epsilon1)}\left[
\frac{1-1/q}{2^ke}\right]^{1/2},
\]
which is greater than $q^r/17r^2$ for all possible values of
$q\geq2$ and $r\geq4$.

Finally let $S=\PSL_{r+1}^\epsilon(q)$ for $\epsilon=\pm$. From the
proof of Lemma~\ref{lemma-linear-unitary-bound}, we know that the
minimum centralizer size of an element in $S$ is at least
\[
\frac{q^rH(r,q,\epsilon)}{(r+1,q-\epsilon1)},
\]
where \[H(r,q,+)=\frac{1}{ek}\] with $k:=\min\{x \in\NN: x\geq
\log_q(r+2)\}$ and
\[H(r,q,-)=\left(\frac{q^2-1}{ek'(q+1)^2}\right)^{1/2}\] with
$k':=\min\{x\in \NN: x \text{ odd and } x\geq \log_q(r+2)\}$.
Moreover, by \cite[Theorem 1.1]{Babai13}, the proportion of
$p$-regular elements in $S$ is at least $1/(r+1)$ (note that $p\nmid
q$). We deduce that
\[
k_{p'}(S)\geq \frac{q^rH(r,q,\epsilon)}{(r+1)(r+1,q-\epsilon1)}.
\]
It is straightforward to check that this bound is again larger than
$q^r/(17r^2)$ for all possible $q$ and $r$.
\end{proof}

We remark that it is possible to prove that $k_{p'}(S)>
{q^r}/{(12r^2)}$ for every $S$ but the estimates are a lot more
tedious. Also, when $S$ is an even-dimensional orthogonal group,
there is an explicitly computed constant $c>0$ such that
$k_{p'}(S)>cq^r/r$. By following the proof of
Theorem~\ref{theorem-p-regularagain}, we therefore have:

\begin{theorem}
Let $S$ be a simple group of Lie type defined over the field of $q$
elements with $r$ the rank of the ambient algebraic group. Suppose
that $S$ is not linear or unitary. There exists a universal constant
$c>0$ such that
\[
k_{p'}(S)> \frac{cq^r}{r}
\]
for every prime $p$.
\end{theorem}


\section{p-regular and p'-regular conjugacy
classes}\label{section-p-regular-class}

In this section we prove Theorem~\ref{theorem-general-bound-class}.

We start with an easy observation.

\begin{lemma}\label{lemma-normal-subgroup} Let $G$ be a finite
group and $N \unlhd G$. Then $k_{p}(G/N) \leq k_{p}(G)$ and
$k_{p'}(G/N) \leq k_{p'}(G)$.
\end{lemma}

\begin{proof} Recall that $k_{p'}(G)$ is exactly the number of
$p$-Brauer irreducible characters of $G$ and every character of
$G/N$ can be viewed as a character of $G$. Therefore the inequality
$k_{p'}(G/N) \leq k_{p'}(G)$ follows.

Now let $gN$ be a $p$-element of $G/N$. Suppose that $g=g_p
g_{p'}=g_{p'}g_p$ where $g_p$ is a $p$-element and $g_{p'}$ is a
$p'$-element. Then we have $gN=g_pNg_{p'}N=g_{p'}Ng_pN$ where $g_pN$
is a $p$-element and $g_{p'}N$ is a $p$-regular element of $G/N$.
Since $gN$ is a $p$-element, it follows that $g_{p'}N=N$. Thus
$gN=g_pN$, which means that every $p$-element of $G/N$ has a
representative which is a $p$-element of $G$, proving that
$k_{p}(G/N) \leq k_{p}(G)$.
\end{proof}

Next we improve a key result of \cite{Maroti16}.

\begin{lemma}\label{lemma-Maroti}
Let $V$ be an irreducible and faithful $FH$-module for some finite
group $H$ and finite field $F$ of characteristic $p$. Suppose that
$p$ does not divide $|H|$. Then we have $k(H) + n(H, V ) - 1 \geq 2
\sqrt{p-1}$ with equality if and only if $\sqrt{p - 1}$ is an
integer, $|V | = |F| = p$ and $|H| = \sqrt{p - 1}$.
\end{lemma}

\begin{proof}
This follows from \cite[Theorem 2.1]{Maroti16} for $p \geq 59$. We take this opportunity to note that \cite[Lemma 3.2]{Maroti16} should be replaced by a different but similar statement, namely by ``With the above notation and assumptions, $$\max \{ t+1, k  \} \leq \binom{t+k-1}{k-1} \leq n(G,V)."$$ (in the notation of \cite{Maroti16}). The statement is \cite[Lemma 2.6]{Foulser}. This does not effect the proof of \cite[Theorem 2.1]{Maroti16}, only straightforward and minor changes are to be made.

Assume that $p < 59$. For convenience, let $f(H,V) = k(H) + n(H,V) - 1$.

If $|V| = p$, then $H$ is cyclic of order dividing $p-1$ and
$$f(H,V) = |H| + \frac{p-1}{|H|} \geq 2 \sqrt{p-1}$$ with equality
if and only if $\sqrt{p-1}$ is an integer and $|H| = \sqrt{p-1}$.
From now on assume that $|V| > p$.

Let $|V| = p^{2}$. Assume first that $H$ is solvable. The argument
of H\'ethelyi, K\"ulshammer \cite[p. 661-662]{HK} gives $f(H,V) \geq
(49p+1)/60$. It is easy to see that $(49p+1)/60 > 2 \sqrt{p-1}$
unless $p \in \{ 2, 3 \}$. Let $p \in \{ 2,3 \}$. We are finished if
$k(H) \geq 3 > 2 \sqrt{p-1}$. Otherwise $|H| \leq 2$ and the integer
$f(H,V)$ is at least $1 + (p^{2}-1)/2 \geq 5/2$, and so $f(H,V) \geq
3 > 2 \sqrt{p-1}$.

Assume now that $|V| = p^{2}$ and $H$ is non-solvable. Then $H \leq
\bZ(\GL(V)) \cdot \SL(V)$ by \cite[Theorem 3.5]{Giudici} and so
$H/\bZ(H)$ is isomorphic to $A_5$ and $p \in \{ 5, 11, 31, 41 \}$ by
\cite[p. 213-214]{Huppert} or \cite[p. 285]{Dickson}. Moreover,
since $|\bZ(\SL(V))| = 2$, the factor group $H/ (\bZ(\SL(V)) \cap
H)$ is a direct product of $A_5$ and a cyclic group of order (at
least) $|\bZ(H)|/2$. This implies that $k(H) \geq 2.5 \cdot
|\bZ(H)|$. We thus have
\begin{equation}
\label{e1} f(H,V) \geq 2.5 \cdot |\bZ(H)| + \frac{p^{2}-1}{60 \cdot
|\bZ(H)|} = 2.5 \cdot |\bZ(H)| + \frac{(2.5/60) \cdot (p^{2}-1)}{2.5
\cdot |\bZ(H)|}.
\end{equation}
The right-hand side of (\ref{e1}) is at least $2 \sqrt{(2.5/60)
\cdot (p^{2}-1)} > 0.4 \sqrt{p^{2}-1}$, which is larger than $2
\sqrt{p-1}$ unless $p \in \{ 5, 11 \}$. If $p = 5$, then $H =
\SL(2,5) = \SL(V)$ and $f(H,V) = 10 > 2 \sqrt{p-1}$. Assume that $p
= 11$. If $\bZ(H)$ is non-trivial, then $k(H) \geq 9$ and so $f(H,V)
\geq 7 > 2 \sqrt{10-1}$. If $H$ is isomorphic to $A_5$ (a case which
probably does not occur), then $k(H) = 5$ and $(11^{2}-1)/|H| = 2$
and so $f(H,V) \geq 7$.

Assume that $|V| \geq p^{3}$. Let $c = (p-1)/(p^{3}-1)$. If $k(H) >
c \cdot |H|$, then
$$f(H,V) > c \cdot |H| + \frac{p^{3}-1}{|H|} = c \cdot |H| +
\frac{c \cdot (p^{3}-1)}{c \cdot |H|} \geq 2 \sqrt{c \cdot
(p^{3}-1)} = 2 \sqrt{p-1}.$$ Thus assume that $k(H) \leq c \cdot
|H|$.

Observe that $c \leq 1/7$. The list of finite groups $X$ with $k(X)
\leq 4$ found in \cite{VLVL1} shows that $k(X) > |X|/7$. We may thus
assume that $k(H) \geq 5$.

If $p \leq 7$, then $f(H,V) \geq 5 + 1 > 2 \sqrt{7-1} \geq 2
\sqrt{p-1}$. We may have $p \geq 11$.

Observe that $c \leq 1/133$. The list of finite groups $X$ with
$k(X) \leq 8$ found in \cite{VLVL1} shows that $k(X) > |X|/133$. We
may thus assume that $k(H) \geq 9$.

If $p \leq 23$, then $f(H,V) \geq 9 + 1 > 2 \sqrt{23-1} \geq 2
\sqrt{p-1}$. Assume that $p \geq 29$.

Now $c \leq 1/871$. The list of finite groups $X$ with $k(X) \leq 9$
found in \cite{VLVL1} shows that $k(X) > |X|/871$. We may thus
assume that $k(H) \geq 10$.

If $p \leq 31$, then $f(H,V) \geq 10 + 1 > 2 \sqrt{31-1} \geq 2
\sqrt{p-1}$. Assume that $p \geq 37$.

Let $k(H) = 10$ or $k(H) = 11$. The list in \cite{VLVL1} shows that
$|H| \leq 20160$ in the first case and $|H| \leq 29120$ in the
second. Thus $f(H,V) \geq k(H) + (p^{3}-1)/|H| > 2 \sqrt{p-1}$ for
every prime $p$ with $37 \leq p \leq 53$. Thus $k(H) \geq 12$.

We have $f(H,V) \geq 12 + 1 > 2 \sqrt{p-1}$ for $p \leq 53$, unless
$p = 47$ or $p = 53$. Moreover, if $k(H) \geq 14$, then $f(H,V) \geq
14 + 1 > 2 \sqrt{53-1} \geq 2 \sqrt{p-1}$. Thus we may assume that
$(k(H),p) \in \{ (12,47), (12,53), (13,47), (13,53) \}$.

Let $k(H) = 12$. Then $|H| \leq 43320$ or $H$ is isomorphic to the
Mathieu group $M_{22}$ by \cite{VLVL2}. In the former case $f(H,V)
\geq 12 + (p^{3}-1)/43320 > 2 \sqrt{p-1}$. Observe that $|M_{22}|$
is equal to $443520$, which does not divide $|\GL(3,p)|$ (for $p \in
\{ 47, 53 \}$). Thus in the second case $f(H,V) \geq 12 +
(p^{4}-1)/|H| > 23 > 2 \sqrt{p-1}$.

Finally, let $k(H) = 13$ and $p \in \{ 47, 53 \}$. If $p = 47$, then
$f(H,V) \geq 14$ which is larger than $2 \sqrt{p-1}$. Let $p = 53$.
If $H$ is not a nilpotent group, then the list in \cite{VLS} shows
that $|H| \leq 4840$ and so $f(H,V) \geq 13 + (53^{3}-1)/4840 > 2
\sqrt{53-1}$. Let $H$ be nilpotent. Since $13 = k(H) =
\prod_{i=1}^{t} k(P_{i})$ where $P_{i}$ is a Sylow $p_{i}$-subgroup
of $H$ and $\{ p_{1}, \ldots , p_{t} \}$ is the set of distinct
prime divisors of $|H|$ and since $13$ is prime, we must have $t =
1$ and that $H$ is a $p_{1}$-group. Now $H$ cannot be transitive on
$V \setminus \{ 0 \}$ since $52 = (p-1) \mid (|V|-1)$ cannot divide
$|H|$. This means that $f(H,V) \geq 13 + 2 > 2 \sqrt{53-1}$. The
proof is complete.
\end{proof}

We finally can prove Theorem \ref{theorem-general-bound-class},
which is restated below.

\begin{theorem}
Let $p$ be a prime and $G$ be a finite group of order divisible by
$p$. We have
\[
k_{p}(G)+k_{p'}(G)\geq 2\sqrt{p-1}.
\]
Moreover, the equality occurs if and only if $\sqrt{p - 1}$ is an
integer, $G = C_p \rtimes C_{\sqrt{p-1}}$ and $\bC_G(C_p) = C_p$.
\end{theorem}

\begin{proof}
If $\sqrt{p-1}$ is an integer, $G = C_p \rtimes C_{\sqrt{p-1}}$ and
$\bC_G(C_p) = C_p$, then $$k_{p}(G) + k_{p'}(G) = 2 \sqrt{p-1}.$$
Assume that $G$ is different from the group $G = C_p \rtimes
C_{\sqrt{p-1}}$ with $\bC_G(C_p) = C_p$ when $\sqrt{p-1}$ is an
integer. We proceed to show by induction on the size of $G$ that
$k_{p}(G) + k_{p'}(G) > 2 \sqrt{p-1}$.

This is clearly true in case $G$ is a cyclic group of order $p$
(different from $C_{2}$). If $G$ is an almost simple group, the
claim follows from Theorem \ref{proposition-simple-groups} unless
$(\mathrm{Soc}(G),p) = (\Al_{5},5)$ or $(\mathrm{Soc}(G),p) =
(\PSL_{2}(16),17)$. Even in these two exceptional cases the bound
can be checked using \cite{Atl1}.

Let $N$ be a non-trivial normal subgroup of $G$. We have
$$k_{p}(G/N) + k_{p'}(G/N) \leq k_{p}(G) + k_{p'}(G)$$ by Lemma
\ref{lemma-normal-subgroup}. We may assume by induction that $p
\nmid |G/N|$, or $\sqrt{p-1}$ is an integer, $G/N = C_p \rtimes
C_{\sqrt{p-1}}$ with $\bC_{G/N}(C_p) = C_p$ and $$2 \sqrt{p-1} =
k_{p}(G/N) + k_{p'}(G/N) \leq k_{p}(G) + k_{p'}(G).$$ In this latter
case we are finished by Lemma \ref{lemma-normal-subgroup} unless
$k_{p}(G/N) = k_{p}(G)$ and $k_{p'}(G/N) = k_{p'}(G)$. However,
$k_{p}(G/N) < k_{p}(G)$ if $p \mid |N|$ and $k_{p'}(G/N) <
k_{p'}(G)$ if $p \nmid |N|$. We are thus left with the case that $p
\nmid |G/N|$ and $p \mid |N|$. Since $p \nmid |G/N|$ and $p \mid
|N|$ hold for every non-trivial normal subgroup $N$ of $G$, the
group $G$ must have a unique minimal normal subgroup $V$. Moreover,
$G$ has a complement $H$ for $V$ by the Schur-Zassenhaus theorem.

Assume that $V$ is elementary abelian. The subgroup $H$ of $G$ acts
faithfully, coprimely and irreducibly on $V$. We have $k(H) + n(H, V
) - 1 > 2 \sqrt{p-1}$ by Lemma~\ref{lemma-Maroti}. Observe that
$k_{p'}(G) \geq k_{p'}(G/V)=k_{p'}(H)=k(H)$ and that $k_p(G)\geq
n(H,V)-1$ since each $H$-orbit on $V$ produces a $G$-conjugacy class
of $p$-elements. These give the desired $k_{p'}(G) + k_{p}(G) > 2
\sqrt{p-1}$ bound.

It remains to assume that $V$ is non-abelian and thus it is
isomorphic to a direct product of copies of a non-abelian simple
group $S$. Since almost simple groups $G$ have been treated before,
$V$ is the direct product of at least two copies of $S$. As $p\mid
|V|$, we have $p\mid |S|$. First suppose that $(S,p)$ is neither
$(\Al_5,5)$ nor $(\PSL_2(16),17)$. From
Theorem~\ref{proposition-simple-groups}, we know that $H$ has more
than $2\sqrt{p-1}$ orbits on conjugacy classes of $p$-regular and
$p'$-regular elements of $S$, and therefore of $V$. Clearly, the
number of these orbits is at most $k_{p}(G)+k_{p'}(G)$, and hence
the theorem follows. Even in the case
$(S,p)\in\{(\Al_5,5),(\PSL_2(16),17)\}$ we are also done since the
number of $G$-orbits on $p$-regular classes of $V$ is at least
$k(k+1)/2$, where $k=n(\Aut(S),\Cl_{p'}(S))=4$ for $(S,p)=(\Al_5,5)$
and $5$ for $(S,p)=(\PSL_2(16),17)$. We have finished the proof.
\end{proof}

\section{The number of Brauer characters of non-$p$-solvable
groups}\label{section-Brauer-characters}

In this section we prove Theorem \ref{nonsolvable}. Let $p$ be a
prime. The set of irreducible $p$-Brauer characters of a finite
group $G$ is denoted by $\IBR_p(G)$. We give two lower bounds for
$|\IBR_p(G)|$ in case $G$ is a non-$p$-solvable finite group. Our
result can be compared to \cite[Theorem 1.1]{Moreto-Nguyen16} where
it was shown that $|\IBR_p(G)|$ is bounded below by a function of
$|G/\bO_\infty(G)|$ where $\bO_\infty(G)$ denotes the largest
solvable normal subgroup of $G$.

Let $G$ be a non-$p$-solvable finite group. Let $N:=\bO_\infty(G)$.
We have $|\IBR_p(G)| = k_{p'}(G) \geq k_{p'}(G/N) = |\IBR_p(G/N)|$
by Lemma~\ref{lemma-normal-subgroup}. It is sufficient to establish
the bounds for the group $G/N$, that is, we may assume that $G$ has
no elementary abelian minimal normal subgroup. We may also assume by
the same argument that every minimal normal subgroup of $G$ has
order divisible by $p$.

Let $\mathrm{Soc}(G)$ denote the socle of $G$ defined to be the
product of all minimal normal subgroups of $G$. In this case this is
a characteristic subgroup which is a direct product of non-abelian
simple groups. Let $S$ be a non-trivial direct summand of
$\mathrm{Soc}(G)$.

Assume first that $S$ is $G$-invariant. Observe that $k_{p'}(G)$ is
at least the number of $G$-orbits of $p$-regular elements in $S$.
This latter number is greater than $2 \sqrt{p-1}$ by (iii) of
Theorem \ref{proposition-simple-groups}, unless $S$ and $p$ appear
in Table \ref{table-exceptions} and thus $p \leq 257$. In any case,
$k_{p'}(G) > 2 \sqrt{p-1}$.

We are left with the case when $S$ is not $G$-invariant. Let $k =
n(\Aut(S),\mathrm{Cl}_{p'}(S))$. Again by (iii) of Theorem
\ref{proposition-simple-groups}, this is larger than $2 \sqrt{p-1}$
unless possibly if $p \leq 257$, but in any case $k > \sqrt{p-1}$.
Let $t$ denote the number of different conjugates of $S$ under $G$.
Then
$$k_{p'}(G) \geq n(G, \mathrm{Cl}_{p'}(\mathrm{Soc}(G))) \geq \binom{k+t-1}{t} \geq \frac{k(k+1)}{2}.$$
If $p > 257$, then $k(k+1)/2 > 2(p-1) > 2 \sqrt{p-1}$. If $p \leq
257$, then $$k(k+1)/2 > (p-1)/2 \geq \sqrt{p-1},$$ unless $p \leq
3$. Finally, assume that $p \leq 3$. The group $G$ has at least
three different prime divisors by Burnside's Theorem. Thus
$k_{p'}(G) \geq 3 > \sqrt{p-1}$.


\section{p-rational and p'-rational
characters}\label{section-p-rational}

In this section we prove
Theorem~\ref{theorem-general-bound-character}.

We first prove Theorem \ref{theorem-general-bound-character} for
$p$-solvable groups. In fact, we can do a bit more. The following
implies Theorem~\ref{theorem-general-bound-character} for
$p$-solvable groups. Here $\QQ_p$ denotes the cyclotomic extension
of rational numbers by a primitive $p$th root of unity. Also, we use
the standard notation $\QQ(\chi)$ for the field of values of a
character $\chi$.

\begin{theorem}\label{theorem-p-rat-solvable}
Let $G$ be a finite $p$-solvable group of order divisible by $p$.
Then
\[
|\Irr_{\mathrm{p-rat}}(G)\cup\Irr_{\QQ_p}(G)|\geq 2\sqrt{p-1}.
\]
\end{theorem}

\begin{proof}
By induction we may assume that for every minimal normal subgroup
$N$ of $G$ we have $p \nmid |G/N|$. It follows that $G$ has a unique
minimal normal subgroup $V$ and $p \nmid |G/V|$. The group $V$ has a
complement $H$ in $G$ by the Schur-Zassenhaus theorem. The group $V$
can be viewed as an irreducible and faithful $FH$-module where $F$
is a finite field $F$ of characteristic $p$. It follows by
Lemma~\ref{lemma-Maroti} that
\[
k(H) + n(H, V ) - 1 \geq 2 \sqrt{p-1}.
\]

Note that every irreducible character of $H$, viewed as a character
of $G$, has values in $\QQ_{|H|}$, and hence is $p$-rational.
Therefore $G$ has exactly $k(H)$ $p$-rational characters whose
kernels contain $V$.

We claim that $G$ has at least $m:=n(H,V)-1$ irreducible characters
with values in $\QQ_p$ and their kernels do not contain $V$. Note
that all characters of $V$ have values in $\QQ_p$.

Let $\theta_1,\theta_2, \ldots ,\theta_{m}$ be representatives of the
$H$-orbits on $\Irr(V)\backslash \{\textbf{1}_H\}$. For each $1\leq
i\leq m$, the character $\theta_i$ has a \emph{canonical extension}
to $I_G(\theta_i)$, say $\widehat{\theta_i}$ such that
$\QQ(\widehat{\theta_i})=\QQ(\theta_i)\subseteq \QQ_p$ (see
\cite[Corollary 6.4]{Navarro18} for instance). It follows that
$\QQ(\widehat{\theta_i}^G)\subseteq \QQ_p$. Also, by Clifford's
theorem we have $\widehat{\theta_i}^G\in\Irr(G)$. Note that the
$\widehat{\theta_i}^G$ are pairwise different. Therefore the claim
follows.

Now we have \[ |\Irr_{\QQ_p}(G)\cup\Irr_{\mathrm{p-rat}}(G)|\geq k(H) + n(H,
V ) - 1 \geq 2 \sqrt{p-1},
\]
which proves the theorem.
\end{proof}


\begin{lemma}\label{lemma-p=2}
Let $G$ be a nonsolvable group. Then $|\Irr_{2-rat}(G)|\geq 3$.
Consequently, Theorem~\ref{theorem-general-bound-character} holds
for $p=2$.
\end{lemma}

\begin{proof} By modding out a solvable normal subgroup if
necessary, we assume that $G$ has a nonabelian minimal normal
subgroup $N$, which is a direct product of copies of a nonabelian
simple group $S$. By \cite[Lemma 4.1]{HSTR20}, there exists a
non-principal character $\theta\in\Irr(S)$ that is extendible to a
rational-valued character of $\Aut(S)$, and therefore $G$ has a
rational irreducible character $\chi$ which extends $\theta\times
\cdots \times \theta\in\Irr(N)$.

If $G/N$ has even order, then by Burnside's theorem it has a
nontrivial rational irreducible character, and together with $\chi$
above and the trivial character, it follows that $|\Irr_\QQ(G)|\geq
3$, as wanted. If $|G/N| > 1$ is odd, then every
$\varphi\in\Irr(G/N)$ is $2$-rational and thus all the characters of
the form $\chi\varphi\in\Irr(G)$ are $2$-rational, implying that
$|\Irr_{2-rat}(G)|\geq 3$.

We now can assume that $G=N$. It is in fact sufficient to show that
$|\Irr_{2-rat}(S)|\geq 3$ for every nonabelian simple group $S$.
This is easy to check when $S$ is a sporadic group, the Tits group,
$\PSL_2(q)$ with $q\in\{5, 7, 8, 9, 17\}$, $\PSL_3(3)$, $\PSU_3(3)$,
or $\PSU_4(2)$ using \cite{Atl1}. It is also easy for $S=\Al_n$ by
considering the restrictions of the irreducible characters of $S$
labeled by the partitions $(n-1,1)$ and $(n-2,2)$. So we can, and we
will, assume that $S$ is not one of these groups. First note that
the trivial and Steinberg characters of $S$ are rational. We claim
that $S$ has a $2$-rational semisimple character, and thus the
required bound follows.

By the classification, we can find a simple algebraic group $\GC$ of
adjoint type and a Frobenius endomorphism $F:\GC\rightarrow \GC$
such that $S=[G,G]$ for $G:=\GC^F$. Let $(\GC^\ast,F^\ast)$ be dual
to $(\GC,F)$ and let $G^\ast:={\GC^\ast}^{F^\ast}$. By Lusztig's
classification of the complex irreducible characters of finite
reductive groups \cite{Digne-Michel}, each $G^\ast$-conjugacy class
$s^{G^\ast}$ of a semisimple element $s\in G^\ast$ corresponds to a
semisimple character $\chi_s\in\Irr(G)$. This $\chi_s$ has values in
$\QQ_{|s|}$ by \cite[Lemma 4.2]{Giannelli-Hung-Schaeffer-Rodriguez}
and moreover, by \cite[Proposition 5.1]{Tiep}, if $|s|$ is coprime
to $|\bZ(G^\ast)|$ then $\chi_s$ restricts irreducibly to $S$.

From the assumption on $S$, we have that $|G^\ast|$ is divisible by
at least three different odd primes, and thus $G^\ast$ always
possesses a semisimple element $s$ such that $(|s|,
2|\bZ(G^\ast)|)=1$. This $s$ then corresponds to a semisimple
character $\chi_s\in \Irr(G)$ such that $\chi_s$ restricts
irreducibly to $S$ and $\QQ(\chi_s)\subseteq \QQ_{|s|}$. Thus
$(\chi_s)_{S}$ is $2$-rational.

Theorem~\ref{theorem-general-bound-character} follows for $p=2$
since the solvable case was already treated in
Theorem~\ref{theorem-p-rat-solvable}.
\end{proof}

%
%

The following observation is crucial in the proof of
Theorem~\ref{theorem-general-bound-character} for odd $p$. It is
well-known but we could not find a reference.

\begin{lemma}\label{lemma-p-rat-regular}
For every finite group $X$ and odd prime $p$, $|\Irr_{\mathrm{p-rat}}(X)|\geq
k_{p'}(X)$.
\end{lemma}

\begin{proof} The lemma is obvious when
$p\nmid |X|$. So we assume $p\mid |X|$. Consider the natural actions
of $\Gamma:=\mathrm{Gal}(\QQ_{|X|}/\QQ_{|X|_{p'}})\cong \mathrm{Gal}(\QQ_{|X|_p}/\QQ)$
on classes and irreducible characters of $X$. Note that $\Gamma$ is
cyclic of order $|X|_p(p-1)/p$. Let $\xi$ be a generator of
$\Gamma$. By Brauer's permutation lemma, $\xi$ fixes the same number
of irreducible characters and classes. Each irreducible character
fixed by $\xi$ has values in $\QQ_{|X|_{p'}}$ and therefore
$p$-rational. On the other hand, for each conjugacy class $\Cl(g)$
of a $p$-regular element $g$, we have $\chi(g)\in \QQ_{|X|_{p'}}$
for all $\chi\in\Irr(X)$, implying that $\Cl(g)$ is fixed by $\xi$.
The lemma now follows.
\end{proof}

Using the theory of the so-called $B_p$-characters \cite{Isaacs84},
one can similarly show that, for a $p$-solvable group $G$,
$|\Irr_{\mathrm{p'-rat}}(G)|$ is no less than the number of classes of
$p$-elements. This seems to be true for all finite groups but
remains to be confirmed.

\begin{lemma}\label{lemma-p-rat-numberfactor>2}
Let $G$ be a finite group with a non-abelian normal subgroup
$$N\cong S\times \cdots \times S,$$ where $S$ is simple, $2<p\mid
|S|$, and there are at least two factors of $S$ in $N$. Then
$|\Irr_{\mathrm{p-rat}}(G)|>2\sqrt{p-1}$.
\end{lemma}

\begin{proof} Let $k:=n(\Aut(S),\Cl_{p'}(S))$. Since there are at least two factors
of $S$ in $N$, as before we have $n(\Aut(N),\Cl_{p'}(N))\geq
{k(k+1)}/{2}$. Since $k\geq2(p-1)^{1/4}$ by
Theorem~\ref{proposition-simple-groups}(ii) and $G$ acts naturally
on $\Cl_{p'}(N)$, it follows that $n(G,\Cl_{p'}(N))>2\sqrt{p-1}$,
which in turn implies that $k_{p'}(G)>2\sqrt{p-1}$. The lemma now
follows by Lemma~\ref{lemma-p-rat-regular}.
\end{proof}

\begin{theorem}\label{theorem-p-rat-almostsimple}
Let $S$ be a nonabelian simple group of order divisible by $p>2$ and
$S\leq G\leq \Aut(S)$ be an almost simple group. Then
$$|\Irr_{\mathrm{p-rat}}(G)\cup \Irr_{\mathrm{p'-rat}}(G)|> 2\sqrt{p-1}.$$
\end{theorem}

\begin{proof} Suppose first that $S$ is not listed in
Table~\ref{table-exceptions}. Then by
Theorem~\ref{proposition-simple-groups}(iii) we have
$n(\Aut(S),\Cl_{p'}(S))>2\sqrt{p-1}$. It follows that
$k_{p'}(G)>2\sqrt{p-1}$, implying that $\Irr_{\mathrm{p-rat}}(G)>2\sqrt{p-1}$
by Lemma~\ref{lemma-p-rat-regular}.

We now go over the simple groups in Table~\ref{table-exceptions} and
establish the bound for each of them. Indeed we are able to check
most of them directly using \cite{Atl1,GAP}, except the ones below.

Let $(S,p)=(\PSU_3(16),241)$. In the proof of Lemma~\ref{PSU3}, we
have shown that $k_{p'}(S)>2(16^2-16+1)/3>160$. Therefore, if
$|G/S|\leq 4$ then $k_{p'}(G)>160/4>2\sqrt{p-1}$ and we are done. It
remains to assume that $G=\Aut(S)=S\rtimes C_8$. Again in the proof
of Lemma~\ref{PSU3}, we already showed that $G=\Aut(S)$ has at least
27 orbits on $p$-regular classes of $S$, and so
$k_{p'}(G)\geq n(G,\Cl_{p'}(S))+k_{p'}(G/S)-1\geq 27+7=34$. We now
have $\Irr_{\mathrm{p-rat}}(G)>2\sqrt{p-1}$, as desired.


Let $S=\ta B_2(128)$ and $p=113$ or $127$. If $G=S$ then we have
$k_{p'}(G)>2\sqrt{p-1}$ as analyzed in
Section~\ref{section-exceptional} (1). So suppose $G>S$ and thus
$G=\Aut(S)=S\rtimes C_7$. First we note that the trivial and
Steinberg characters of $S$ have rational extensions to $\Aut(S)$.
Also, $S$ has a rational class of elements of order $5$ that is
$\Aut(S)$-invariant, this semisimple class corresponds to a rational
semisimple character of $S$ of odd degree, which therefore has a
rational extension to $\Aut(S)$ as well. By Gallagher's theorem, we
obtain 21 irreducible characters of $G$ with values in $\QQ_7$.

As mentioned in Section~\ref{section-exceptional}, $\Aut(S)$ has
four orbits of size 7 on classes of elements of order 145. One
(semisimple) element in such an orbit produces an irreducible
semisimple character of $S$ with values in $\QQ_{145}$ and moreover
has $S$ as the stabilizer group in $\Aut(S)$, and thus gives rise to
1 irreducible character of $\Aut(S)$ with values in
$\QQ(\chi)\subseteq\QQ_{145}$, by Clifford's theorem. We now have 4
more irreducible $p$-rational characters of $G=\Aut(S)$, different
from the 21 characters produced in the previous paragraph. We have
shown that $G$ has at least $25$, which is larger than $2\sqrt{p-1}$,
$p$-rational irreducible characters.

For $(S,p)=(\Omega_8^-(4),257)$ we know that $S$ has exactly
$32=2\sqrt{p-1}$ different element orders coprime to $p$ (these
orders are listed in the proof of Lemma~\ref{lemma-POmega8-}), and
so clearly $k_{p'}(G)>32$ if $G\neq S$ since there is at least one
$p$-regular class of $G$ outside $S$. In fact we still have
$k_{p'}(G)>32$ when $G=S$ since $S$ has at least four unipotent
classes by Lemma~\ref{lemma-unipotent-class} and at least 29
semisimple classes coming from 29 different odd element orders
coprime to $p$.
\end{proof}

We are now in position to prove
Theorem~\ref{theorem-general-bound-character} for all finite groups.

\begin{theorem}
Let $G$ be a finite group and $p$ a prime divisor of $|G|$. Then
\[
|\Irr_{\mathrm{p-rat}}(G)\cup\Irr_{\mathrm{p'-rat}}(G)|\geq 2\sqrt{p-1}.
\]
Moreover, the equality occurs if and only if $\sqrt{p - 1}$ is an
integer, $G = C_p \rtimes C_{\sqrt{p-1}}$ and $\bC_G(C_p) = C_p$.
\end{theorem}

\begin{proof}
First we prove the inequality. The case $p=2$ is done by
Lemma~\ref{lemma-p=2}, so we will assume that $p$ is odd. We proceed
in the same way as in the proof of
Theorem~\ref{theorem-p-rat-solvable} to come up with the situation
where $G$ has a unique minimal normal subgroup $N$ of order
divisible by $p$ such that $p\nmid |G/N|$. If $N$ is abelian then we
are done by Theorem~\ref{theorem-p-rat-solvable}, so we assume
furthermore that $N$ is nonabelian, which means that $N$ is
isomorphic to a direct product of say $k$ copies of a nonabelian
simple group $S$.

If $k \geq 2$, then we are done by
Lemma~\ref{lemma-p-rat-numberfactor>2}. On the other hand, if $N=S$
then $G$ is an almost simple group with socle $S$, and thus we are
done as well by Theorem~\ref{theorem-p-rat-almostsimple}. This
completes the proof of the first part of the theorem.

We now move on to prove the second part. If $\sqrt{p-1}$ is an
integer, $G = C_{p} \rtimes C_{\sqrt{p-1}}$ and $\bC_G(C_p) = C_p$,
then
$$|\Irr_{\mathrm{p-rat}}(G)\cup\Irr_{\mathrm{p'-rat}}(G)| = |\Irr(G)| = 2 \sqrt{p-1}.$$
Assume that $G$ is different from the group $C_{p} \rtimes
C_{\sqrt{p-1}}$ with $\bC_G(C_p) = C_p$ in case $\sqrt{p-1}$ is an
integer. We proceed to prove by induction on the size of $G$ that
$|\Irr_{\mathrm{p-rat}}(G)\cup\Irr_{\mathrm{p'-rat}}(G)| > 2\sqrt{p-1}$. This is clear
in case $G$ is a cyclic group of order $p$ (excluding the case
$p=2$).

Let $N$ be a minimal normal subgroup of $G$. We have
$\Irr_{\mathrm{p-rat}}(G/N) \subseteq \Irr_{\mathrm{p-rat}}(G)$, $\Irr_{\mathrm{p'-rat}}(G/N)
\subseteq \Irr_{\mathrm{p'-rat}}(G)$ and
$$\Irr_{\mathrm{p-rat}}(G/N)\cap\Irr_{\mathrm{p'-rat}}(G/N) = \Irr_{\QQ}(G/N) \subseteq
\Irr_{\QQ}(G) = \Irr_{\mathrm{p-rat}}(G)\cap\Irr_{\mathrm{p'-rat}}(G).$$ We may assume
by induction that $p \nmid |G/N|$, or that $\sqrt{p-1}$ is an
integer, $G/N = C_{p} \rtimes C_{\sqrt{p-1}}$ with $\bC_{G/N}(C_p) =
C_p$.

In fact the case $p \nmid |G/N|$ is done by
Lemma~\ref{lemma-p-rat-numberfactor>2} and
Theorem~\ref{theorem-p-rat-almostsimple} when $N$ is non-abelian and
by Lemma~\ref{lemma-Maroti} when $N$ is abelian.

Assume that the latter case holds. First suppose that $N$ is an
elementary abelian $r$-group. It then may be viewed as an
irreducible $G/N$-module. If the cyclic normal subgroup $C_p$ of
$G/N$ acts fixed-point-freely on $N$ and thus also on $\Irr(N)$,
then there must be at least $1$ $p$-rational irreducible character
of $G$ by Clifford's theorem which does not contain $N$ in its
kernel. This proves the desired bound. Assume that the cyclic normal
subgroup $C_p$ of $G/N$ has a non-trivial fixed point on $N$. In
this case $|N| = r$ and $G$ contains an abelian normal subgroup $M$
of order $rp$. Just as before, there is at least $1$ $p$-rational
irreducible character of $G$ by Clifford's theorem which does not
contain $N$ in its kernel. Next we suppose $N$ is non-abelian. As
mentioned in the proof of Lemma~\ref{lemma-p=2}, $N$ has an
irreducible character that is extendible to a rational character of
its inertia subgroup in $G$, and thus producing a rational
irreducible character of $G$. In either case we always have
\begin{align*}|\Irr_{\mathrm{p-rat}}(G)\cup\Irr_{\mathrm{p'-rat}}(G)| &\geq
|\Irr_{\mathrm{p-rat}}(G/N)\cup\Irr_{\mathrm{p'-rat}}(G/N)|+
|\Irr_{\mathrm{p-rat}}(G|N)|\\&\geq 2\sqrt{p-1}+1,\end{align*} as wanted. The
proof is completed.
\end{proof}

We conclude by remarking that although $|\Irr_{\mathrm{p-rat}}(G)|\geq
k_{p'}(G)$ by Lemma~\ref{lemma-p-rat-regular} and
$|\Irr_{\mathrm{p'-rat}}(G)|$ is conjecturally at least $1+k_p(G)$ (see the
discussion after Lemma~\ref{lemma-p-rat-regular}), it does not
follow that $|\Irr_{\mathrm{p-rat}}(G)\cup\Irr_{\mathrm{p'-rat}}(G)|\geq
k_{p'}(G)+k_p(G)$, as $\Irr_{\mathrm{p-rat}}(G)$ and $\Irr_{\mathrm{p'-rat}}(G)$ have
those rational characters, including the trivial character, in
common. However, at the time of this writing, we have not found a
counterexample yet.


\end{document}